\documentclass{amsart}
\usepackage{amsmath,amssymb,amsthm,amsfonts,enumerate,array,color,lscape,fancyhdr,layout,pst-all}

\usepackage[english]{babel}
\usepackage[latin1]{inputenc}
\usepackage{young}

\usepackage[hcentermath]{youngtab}

\usepackage{graphicx}
\usepackage{float}
\usepackage{pstricks}

\newcounter{numberofremark}
\newcommand\nothing[1]{}
\usepackage[active]{srcltx}
\usepackage[plainpages]{hyperref}

\newcommand{\dcl}{\DeclareMathOperator}
\dcl\cdet{cdet} \dcl\Sp{Specm} \dcl\depth{depth} \dcl\im{Im} \dcl\Span{span} \dcl\Ker{Ker} \dcl\Specm{Specm}
\dcl\Supp{Supp} \dcl\codim{codim} \dcl\Y{Y} \dcl\gl{\mathfrak{gl}}    \dcl\U{U} \dcl\T{T}
\dcl\qdet{qdet} \dcl\sgn{sgn} \dcl\gr{gr} \dcl\diag{diag}
\dcl\g{\mathfrak{g}} \dcl\C{\mathbb C} \dcl\dd{{\mathrm d}}

\newcommand\sn{{\mathsf n}}
\newcommand\sm{{\mathsf m}}

\newcommand\D{\mathcal{D}}
\newcommand\Dv{\mathcal{D}^{\bar{v}}}
\newcommand\ev{\mbox{ev}(\bar{v})}
\newcommand\ov{ \bar{v}}

\newcommand\Ga{{\Gamma}}

\newlength\yStones
\newlength\xStones
\newlength\xxStones

\makeatletter
\def\Stones{\pst@object{Stones}}
\def\Stones@i#1{%
  \pst@killglue%
  \begingroup%
  \use@par%
  \setlength\xxStones{\xStones}%
  \expandafter\Stones@ii#1,,\@nil
  \endgroup
  \global\addtolength\xStones{0.6cm}%
  \global\addtolength\yStones{-7.5mm}}%
\def\Stones@ii#1,#2,#3\@nil{%
  \rput(\xxStones,\yStones){%
    \psframebox[framesep=0]{%
      \parbox[c][6mm][c]{11mm}{\makebox[11mm]{$#1$}}}}%
  \addtolength\xxStones{1.2cm}%
  \ifx\relax#2\relax\else\Stones@ii#2,#3\@nil\fi}
\makeatother

\def\Stone#1{\fbox{\makebox[13mm]{\strut#1}}\kern2pt}

\newtheorem{theorem}{Theorem}[section]
\newtheorem{lemma}[theorem]{Lemma}

\newtheorem{proposition}[theorem]{Proposition}

\newtheorem{remark}[theorem]{Remark}

\newtheorem{definition}[theorem]{Definition}

\begin{document}

\title[Gelfand-Tsetlin modules of quantum $\gl_n$]{Gelfand-Tsetlin modules of quantum $\gl_n$ defined by admissible sets of relations}

\author{Vyacheslav Futorny}

\address{Instituto de Matem\'atica e Estat\'istica, Universidade de S\~ao
Paulo,  S\~ao Paulo SP, Brasil} \email{futorny@ime.usp.br,}

\author{Luis Enrique Ramirez}
\address{Universidade Federal do ABC,  Santo Andr\'e SP, Brasil} \email{luis.enrique@ufabc.edu.br,}

\author{Jian Zhang}

\address{Instituto de Matem\'atica e Estat\'istica, Universidade de S\~ao
Paulo,  S\~ao Paulo SP, Brasil} \email{zhang@ime.usp.br,}

\begin{abstract}
The purpose of this paper is to construct new families of irreducible Gelfand-Tsetlin modules for $U_q(\gl_n)$. These modules have arbitrary singularity
and  Gelfand-Tsetlin multiplicities bounded by
$2$. Most previously known irreducible modules had all  Gelfand-Tsetlin multiplicities bounded by
$1$ \cite{FRZ1}, \cite{FRZ2}. In particular, our method works for $q=1$ providing new families of irreducible Gelfand-Tsetlin modules for $\gl_n$. This generalizes the results of \cite{FGR3} and
\cite{FRZ}.

\end{abstract}

\maketitle
\section{Introduction}

 Theory of Gelfand-Tsetlin modules is originated in the classical paper of Gelfand and Tsetlin \cite{GT}, where a basis for all finite dimensional representations of $\gl_n$ was constructed consisting of eigenvectors of certain maximal commutative subalgebra of $U(\gl_n)$, a Gelfand-Tsetlin subalgebra.  Further, infinite dimensional Gelfand-Tsetlin modules for $\gl_n$
were studied  in \cite{GG}, \cite{LP1}, \cite{LP2},  \cite{DFO}, \cite{O}, \cite{FO2}, \cite{Maz1}, \cite{Maz2}, \cite{m:gtsb}, \cite{FGR1}, \cite{FGR2}, \cite{FGR3}, \cite{FGR4}, \cite{FGR5}, \cite{Za}, \cite{RZ}, \cite{Vi1}, \cite{Vi2} among the others. These representations have close connections to different concepts
 in Mathematics and Physics (cf. \cite{KW1}, \cite{KW2},  \cite{GS}, \cite{FM}, \cite{Gr1}, \cite{Gr2}, \cite{CE1}, \cite{CE2}, \cite{FO1} and references therein).

For quantum $\gl_n$
  Gelfand-Tsetlin modules were considered in \cite{MT}, \cite{FH}, \cite{FRZ1}, \cite{FRZ2}. In particular, families of irreducible  Gelfand-Tsetlin modules constructed
in \cite{FRZ1} and \cite{FRZ2} correspond to admissible sets of relations and have all  Gelfand-Tsetlin multiplicities bounded by
$1$. In general, it is difficult to construct explicitly irreducible modules that have  Gelfand-Tsetlin multiplicities bigger than $1$, even in the case of highest weight modules. Some examples can be obtained from \cite{FJMM}. In this paper we combine two methods. First,  the construction of Gelfand-Tsetlin modules out of admissible relations in \cite{FRZ2} and, second,  the construction of $1$-singular modules in \cite{FGR3}. As a result we are able to obtain new families of irreducible Gelfand-Tsetlin modules whose Gelfand-Tsetlin multiplicities bounded by
$2$ (Theorem \ref {thm-when L irred}). Moreover, these modules have arbitrary singularities. The construction is explicit with a basis and the action of the Lie algebra (Theorem \ref{thm-main}). This allows to understand the structure of the constructed modules, in particular to describe the action of the generators of the Gelfand-Tsetlin subalgebra (Theorem \ref{GT module structure}).

Our method also works when $q=1$ and allows to construct new families of irreducible modules for $\gl_n$ (Theorem \ref{thm-gl(n)}). Again, these modules allow arbitrary singularities (generalizing \cite{FGR3}) and  on the other hand have Gelfand-Tsetlin multiplicities bounded by $2$ with a non-diagonalizable action of the Gelfand-Tsetlin subalgebra (generalizing \cite{FRZ}). Moreover,  any irreducible Gelfand-Tsetlin module with a designated singularity appear as a subquotient of the constructed universal module (Theorem \ref{thm-exhaust}).

 \

 \

\noindent{\bf Acknowledgements.}  V.F. is
supported in part by  CNPq  (301320/2013-6) and by
Fapesp  (2014/09310-5).    J. Z. is supported by  Fapesp   (2015/05927-0).

\section{Preliminaries}

Let $P$ be the free $\mathbb Z$-lattice of rank n with the
canonical basis $\{\epsilon_{1},\ldots,\epsilon_{n}\}$, i.e.
$P=\bigoplus_{i=1}^{n}\mathbb Z\epsilon_{i}$, endowed with  symmetric
bilinear form
$\langle\epsilon_{i},\epsilon_{j}\rangle=\delta_{ij}$.
Let $\Pi=\{\alpha_j=\epsilon_j-\epsilon_{j+1}\ |\ j=1,2,\ldots\}$ and
$\Phi=\{\epsilon_{i}-\epsilon_{j}\ |\ 1\leq i\neq j\leq n-1\}$. Then $\Phi$ realizes the root system
of type $A_{n-1}$ with $\Phi$ a basis of simple roots.

 By $U_q$ we denote the quantum enveloping algebra of $\gl_n$.
We define $U_{q}$ as a unital associative algebra generated by $e_{i},f_{i}(1\leq i
\leq n-1)$ and $q^{h}(h\in P)$ with the
following relations:
\begin{align}
q^{0}=1,\  q^{h}q^{h'}=q^{h+h'} \quad (h,h' \in  P),\\
q^{h}e_{i}q^{-h}=q^{\langle h,\alpha_i\rangle}e_{i}  ,\\
q^{h}f_{i}q^{-h}=q^{-\langle h,\alpha_i\rangle}f_{i} ,\\
e_{i}f_{j}-f_{j}e_{i}=\delta_{ij}\frac{q^{\alpha_i}-q^{-\alpha_i}}{q-q^{-1}} ,\\
e_{i}^2e_{j}-(q+q^{-1})e_ie_je_i+e_je_{i}^2=0  \quad (|i-j|=1),\\
f_{i}^2f_{j}-(q+q^{-1})f_if_jf_i+f_jf_{i}^2=0  \quad (|i-j|=1),\\
e_{i}e_j=e_je_i,\  f_if_j=f_jf_i  \quad (|i-j|>1).
\end{align}
%The quantum special linear algebra $U_q(sl_n)$ is the subalgebra of $U_q$
%generated by $e_i,\ f_i,\ q^{\pm \alpha_i}(i=1,2,\ldots, n-1)$.

%The algebra $U_q(gl_n)$ has a Hopf algebra structure with the
%coproduct $\Delta$, counit $\varepsilon$ and the antipode S defined by
%
%\begin{align}
%\Delta(q^h)=q^h\otimes q^h,\\
%\Delta(e_i)=e_i\otimes q^{-\alpha_i}+1\otimes e_i,\\
%\Delta(f_i)=f_i\otimes 1+q^{\alpha_i}\otimes f_k,\\
%\varepsilon(q^h)=1, \varepsilon(e_k)=\varepsilon(f_i)=0,\\
%S(q^h)=q^{-\lambda},\\
%S(e_i)=-e_iq^{\alpha_i},\\
%S(f_i)=-q^{-\alpha_i}f_i.
%\end{align}
%for $h\in P$, $i=1,2,\ldots, n-1$.

Fix the standard Cartan subalgebra  $\mathfrak h$ and the standard triangular decomposition. The weights of $U_q$ will be written as $n$-tuples $(\lambda_1,...,\lambda_n)$.

\begin{remark}[\cite{FRT}, Theorem 12]
$U_{q}$ has the following alternative presentation. It is isomorphic to the algebra generated by
$l_{ij}^{+}$, $l_{ji}^{-}$, $1\leq i \leq j \leq n$ subject to the relations:
\begin{align}
RL_1^{\pm}L_2^{\pm}=L_2^{\pm}L_1^{\pm}R\\
RL_1^{+}L_2^{-}=L_2^{-}L_1^{+}R
\end{align}
where
$R=q\sum_{i}e_{ii}\otimes e_{ii}+\sum_{i\neq j}e_{ii}\otimes e_{jj}
+(q-q^{-1})\sum_{i<j}e_{ij}\otimes e_{ji}$, $e_{ij}\in End(\mathbb{C}^n)$,
$L^{\pm}=(l_{ij}^{\pm})$, $L_1^{\pm}=L^{\pm}\otimes I$ and $L_2^{\pm}=I\otimes L^{\pm}$.
The isomorphism between this two representations is given by
\begin{align*}
l_{ii}^{\pm}=q^{\pm\epsilon_i},\ \ \ \
l_{i,i+1}^{+}=(q-q^{-1})q^{\epsilon_i}e_i,\ \ \ \ l_{i+1,i}^{-}=(q-q^{-1})f_{i}q^{\epsilon_i}.
\end{align*}
\end{remark}

For a commutative ring $R$, by ${\rm Specm}\, R$ we denote the set of maximal ideals of $R$.
For $1\leq j \leq i \leq n$, $\delta^{ij} \in {\mathbb Z}^{\frac{n(n+1)}{2}}$ is defined by  $(\delta^{ij})_{ij}=1$ and all other $(\delta^{ij})_{k\ell}$ are zero. For $i>0$ by $S_i$ we denote the $i$th symmetric group. Let $1(q)$ be the set of all complex $x$ such that $q^{x}=1$. Finally, for any complex number $x$, we set
\begin{align*}
(x)_q=\frac{q^x-1}{q-1},\quad
[x]_q=\frac{q^x-q^{-x}}{q-q^{-1}}.
\end{align*}

\subsection{Gelfand-Tsetlin modules}

Let  for $m\leqslant n$, $\mathfrak{gl}_{m}$ be the Lie subalgebra
of $\gl_n$ spanned by $\{ E_{ij}\,|\, i,j=1,\ldots,m \}$. We have the following chain
\begin{equation*}
\gl_1\subset \gl_2\subset \ldots \subset \gl_n.
\end{equation*}
If we denote by $(U_m)_q$ the quantum universal enveloping algebra of $\gl_m$. We have the following chain $(U_1)_q\subset$ $(U_2)_q\subset$ $\ldots$ $\subset
(U_n)_q$. Let $Z_{m}$ denotes the center of $(U_{m})_{q}$. The subalgebra ${\Ga}_q$ of $U_q$ generated by $\{
Z_m\,|\,m=1,\ldots, n \}$ is called the \emph{Gelfand-Tsetlin
subalgebra} of $U_q$.

\begin{theorem}[\cite{FRT}, Theorem 14]\label{generators of the quantum center}
The center of $U_{q}(\mathfrak {gl}_m)$ is generated by the following $m+1$ elements
$$c_{mk}=\sum_{\sigma,\sigma'\in S_m}(-q)^{l(\sigma)+l(\sigma')}
l_{\sigma(1),\sigma'(1)}^{+}\cdots l_{\sigma(k),\sigma'(k)}^{+}l_{\sigma(k+1),\sigma'(k+1)}^{-}\cdots l_{\sigma(m),\sigma'(m)}^{-},$$
where $0\leq k \leq m$.
\end{theorem}

\begin{definition}
\label{definition-of-GZ-modules} A finitely generated $U_q$-module
$M$ is called a \emph{Gelfand-Tsetlin module (with respect to
$\Ga_q$)} if

\begin{equation}\label{equation-Gelfand-Tsetlin-module-def}
M=\bigoplus_{\sm\in\Sp\Ga_q}M(\sm),
\end{equation}

where $M(\sm)=\{v\in M| \sm^{k}v=0 \text{ for some }k\geq 0\}$.
\end{definition}

Note that $\Gamma_q$ is a Harish-Chandra subalgebra of $U_q$, that is for any $u\in U_q$ the $\Ga_q$-bimodule $\Ga_q u \Ga_q$ is finitely generated left and right $\Ga_q$-module (\cite{MT}, Proposition 1 and \cite{FH}, Proposition 2.8).
As a result we have the following property of Gelfand-Tsetlin modules.

\begin{lemma}\label{lem-cyclic-Gelfand-Tsetlin}
Let $\sm\in\Sp\Ga_q$ and $M$  be  a $U_q$-module generated by a nonzero element $v\in M(\sm)$. Then $M$ is a Gelfand-Tsetlin module.
\end{lemma}
\begin{proof}
Let $M'=\oplus_{\sm\in\Sp\Ga_q} M(\sn)$ be the largest Gelfand-Tsetlin $U_q$-submodule of $M$. We will show that $M'=M$. Indeed, consider any nonzero $u\in U_q$ and apply to $v$. We need to show that $uv\in M'$. Take any nonzero $z\in \Ga_q$. Since $\Ga_q$ is a Harish-Chandra subalgebra and since $\Ga_q v$ is finite dimensional there exists $N$ such that
$$uv, zuv, \ldots,  z^N uv$$ are linearly dependent. Hence, the subspace spanned by $\{uv, zuv, \ldots,  z^{N-1} uv\}$ is $z$-invariant and $\Pi_{i\in I}(z-\gamma_i(z))^N uv=0$ for some scalars $\gamma_i(z)$ and some set $I$. Choose some generators $z_1, \ldots, z_d$ of $\Ga_q$ and
define such scalars $\gamma_{i}(z_j)$, $i\in I_j$ for each generator $z_j$, $j=1, \ldots, d$.
For each element $\bar{i}=\{i_1, i_2, \ldots, \}$ of $I_1 \times I_2 \times \ldots \times I_d$
consider $\sn_{\bar{i}}\in\Sp\Ga_q$ which contains $z_j- \gamma_{i_j}$ for each $j=1, \ldots, d$. Then $$uv\in \sum_{\bar{i}\in I_1 \times I_2 \times \ldots \times I_d} M(\sn_{\bar{I}}),$$ which proves the lemma.
\end{proof}

\section{Gelfand-Tsetlin modules defined by admissible sets of relations}\label{section modules with relations}
In this section we recall the construction of $U_{q}$-modules from \cite{FRZ2}. As particular cases of this construction one obtains  any irreducible finite dimensional module as in \cite{UST2}, Theorem 2.11, and generic modules as in \cite{MT}, Theorem 2.
%\subsection{Definitions}\label{section Definitions}

\begin{definition} For a vector $v=(v_{ij})$ in $\mathbb{C}^{\frac{n(n+1)}{2}}$, by $T(v)$ we will denote the following array with entries $\{v_{ij}:1\leq j\leq i\leq n\}$
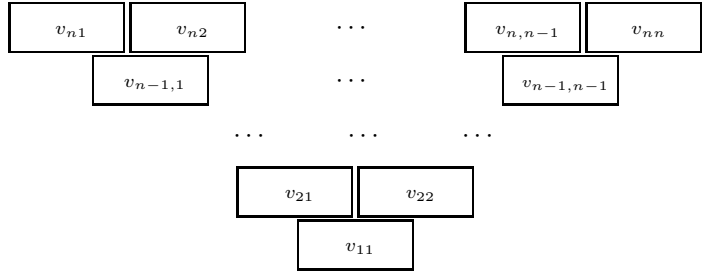
\begin{center}

\Stone{\mbox{ \scriptsize {$v_{n1}$}}}\Stone{\mbox{ \scriptsize {$v_{n2}$}}}\hspace{1cm} $\cdots$ \hspace{1cm} \Stone{\mbox{ \scriptsize {$v_{n,n-1}$}}}\Stone{\mbox{ \scriptsize {$v_{nn}$}}}\\[0.2pt]
\Stone{\mbox{ \scriptsize {$v_{n-1,1}$}}}\hspace{1.5cm} $\cdots$ \hspace{1.5cm} \Stone{\mbox{ \tiny {$v_{n-1,n-1}$}}}\\[0.3cm]
\hspace{0.2cm}$\cdots$ \hspace{0.8cm} $\cdots$ \hspace{0.8cm} $\cdots$\\[0.3cm]
\Stone{\mbox{ \scriptsize {$v_{21}$}}}\Stone{\mbox{ \scriptsize {$v_{22}$}}}\\[0.2pt]
\Stone{\mbox{ \scriptsize {$v_{11}$}}}\\
\medskip
\end{center}
such an array will be called a \emph{Gelfand-Tsetlin tableau} of height $n$.
\end{definition}
Associated with any Gelfand-Tsetlin tableau $T(v)$, by $V(T(v))$ we will denote the $\mathbb{C}$-vector space spanned by the set of Gelfand-Tsetlin tableaux $\mathcal{B}(T(v)):=\{T(v+z)\ |\ z\in\mathbb{Z}^{\frac{n(n-1)}{2}}\}$. For certain subsets $B$ of $\mathcal{B}(T(v))$ and respectively the $\mathbb{C}$-subspaces $V_B(T(v))$ it is possible to define a $U_q$-module structure on $V_B(T(v))$ by the \emph{Gelfand-Tsetlin formulae}:
\begin{equation}\label{Gelfand-Tsetlin formulas}
\begin{split}
q^{\epsilon_{k}}(T(L))&=q^{a_k}T(L),\quad a_k=\sum_{i=1}^{k}l_{k,i}-\sum_{i=1}^{k-1}l_{k-1,i}+k,\ k=1,\ldots,n,\\
e_{k}(T(L))&=-\sum_{j=1}^{k}
\frac{\prod_{i} [l_{k+ 1,i}-l_{k,j}]_q}{\prod_{i\neq j} [l_{k,i}-l_{k,j}]_q}
T(L+\delta^{kj}),\\
f_{k}(T(L))&=\sum_{j=1}^{k}\frac{\prod_{i} [l_{k-1,i}-l_{k,j}]_q}{\prod_{i\neq j} [l_{k,i}-l_{k,j}]_q}T(L-\delta^{kj}),\\
\end{split}
\end{equation}

for all $L\in B$.
%We will study conditions on tableaux $T(v)$ and subsets $B$ of $\mathcal{B}(T(v))$ such that the formulas (\ref{Gelfand-Tsetlin formulas}) define a $U_{q}$-module structure on the $\mathbb{C}$-module %spanned by $B$.

\subsection{Admissible sets of relations}

Set $\mathfrak{V}:=\{(i,j)\ |\ 1\leq j\leq i\leq n\}$. We will consider relations between elements of $\mathfrak{V}$ of the form  $(i,j)\geq (s,t)$ or $(i,j)>(s,t)$. More precisely, we will consider subsets of the following set of relations:
\begin{align}
\mathcal{R} &:=\mathcal{R}^{\geq} \cup \mathcal{R}^{>}\cup \mathcal{R}^{0},
\end{align}
where
\begin{align}
\mathcal{R^{\geq}} &:=\{(i,j)\geq(i-1,j')\ |\ 1\leq j\leq i\leq n,\ 1\leq j'\leq i-1\},\\
\mathcal{R^{>}} &:=\{(i-1,j')>(i,j)\ |\ 1\leq j\leq i\leq n,\ 1\leq j'\leq i-1\},\\
\mathcal{R}^{0} &:=\{(n,i)\geq(n,j)\ |\ 1\leq i\neq j\leq n\}.
\end{align}

%We will denote by $\mathcal{S}$ the set of relations satisfied by a standard tableau. This set of relations is a subset of the all possible differences between elements of consecutive rows, that will be denoted by $\mathfrak{R}$.
%\begin{align*}
%\mathfrak{R}&=\{ l_{ki}-l_{k-1,j}\in\mathbb{Z}_{\geq 0},\ l_{k-1,j'}-l_{k,i'}\in\mathbb{Z}_{> 0}\ |\ 1\leq i,i'\leq k\leq n,1\leq j,j'\leq k-1 \}\\
 %\mathfrak{S}&=\{ l_{ki}-l_{k-1,i}\in\mathbb{Z}_{\geq 0},\ l_{k-1,i}-l_{k,i+1}\in\mathbb{Z}_{> 0}\ |\ 1\leq i\leq k\leq n \}
%\end{align*}

%\begin{definition}
Let $\mathcal{C}$ be a subset of $\mathcal{R}$.
Denote by $\mathfrak{V}(\mathcal{C})$ the set of all $(i,j)$ in $\mathfrak{V}$ such that $(i,j)\geq (\text{respectively }>,\ \leq,\ <)  \ (r, s)\in\mathcal{C}$ for some $(r, s)$.
Let $\mathcal{C}_1$ and $\mathcal{C}_2$ be two subsets of $\mathcal{C}$.
We say that $\mathcal{C}_1$ and $\mathcal{C}_2$ are {\it disconnected}
if
$\mathfrak{V}(\mathcal{C}_1)\cap\mathfrak{V}(\mathcal{C}_2)=\emptyset$,
 otherwise  $\mathcal{C}_1$ and $\mathcal{C}_2$ are connected,
 $\mathcal{C}$ is called {\it decomposable} if it can be decomposed into the  union of two disconnected subsets of $\mathcal{R}$, otherwise $\mathcal{C}$ is called indecomposable.

\begin{definition}
Let $\mathcal{C}$ be any subset of $\mathcal{R}$. Given $(i,j),\ (r,s)\in \mathfrak{V}(\mathcal{C})$ we will write:
\begin{itemize}
\item[(i)] $(i,j)\succeq_{\mathcal{C}} (r,s)$ if, there exists $\{(i_{1},j_{1}),\ldots,(i_{m},j_{m})\}\subseteq \mathfrak{V}(\mathcal{C})$ such that
\begin{align}\label{geq sub C}
\{(i,j)\geq (i_{1},j_1),\  (i_{1},j_1)\geq (i_2,j_2), \cdots,\  (i_{m},j_m)\geq (r,s)\}&\subseteq \mathcal{C}
\end{align}
\item[(ii)] We write $(i,j)\succ_{\mathcal{C}} (r,s)$ if there exists $\{(i_{1},j_{1}),\ldots,(i_{m},j_{m})\}\subseteq \mathfrak{V}(\mathcal{C})$ such that in the condition (\ref{geq sub C}), at least one of the inequalities is $>$.
\end{itemize}
We will say that $(i,j),\ (r,s)$ are $\mathcal{C}$-related if $(i,j)\succeq_{\mathcal{C}} (r,s)$ or $(r,s)\succeq_{\mathcal{C}} (i,j)$. Given another set of relations $\mathcal{C}'$, we say that {\it\bf  $\mathcal{C}$   implies $\mathcal{C}'$} if whenever we have $(i,j)\succ_{\mathcal{C}'} (r,s)$ (respectively $(i,j)\succeq_{\mathcal{C}'} (r,s)$) we also have $(i,j)\succ_{\mathcal{C}} (r,s)$ (respectively $(i,j)\succeq_{\mathcal{C}} (r,s)$). A subset of $\mathcal{R}$ of the form $\{(k,i)\geq (k-1,t),\ (k-1,s)>(k,j)\}$ with $i<j$ and $s<t$, is called a \emph{\bf cross}.
\end{definition}

Now we define the sets of relations of our interest.

\begin{definition}
Let $\mathcal{C}$ be an indecomposable set. We say that $\mathcal{C} $ is {\it \bf admissible}  if it satisfies the following conditions:
\begin{itemize}
\item[(i)] For any $1\leq k-1\leq n$, we have $(k, i)\succ_{\mathcal{C}}(k, j)$ only if $i<j$;
\item[(ii)]  $(n,i)\succeq_{\mathcal{C}} (n,j)$ only if $i<j$;
\item[(iii)] There is no cross in $\mathcal{C} $;
\item[(iv)] For every $(k,i),\ (k,j)\in\mathfrak{V}(\mathcal{C})$ with  $1\leq k\leq n-1$ there exists $s,t$ such that one of the following holds
\begin{equation}\label{condition for admissible}
\begin{split}
 &\{(k,i)>(k+1,s)\geq (k,j),\ (k,i)\geq (k-1,t)>(k,j)\}\subseteq \mathcal{C},\\
 &\{(k,i)>(k+1,s),(k+1,t)\geq (k,j)\}\subseteq  \mathcal{C},\  s<t.
  \end{split}
\end{equation}
\end{itemize}

An arbitrary set $\mathcal{C}$ is admissible if every  indecomposable subset of $\mathcal{C}$ is admissible.
\end{definition}

Denote by $\mathfrak{F}$ the set of all indecomposable admissible subsets. Recall that $1(q):=\{x\in \mathbb{C}\ |\ q^x=1\}$.
 \begin{definition} Let $\mathcal{C}$ be any subset of $\mathcal{R}$ and $T(L)$ any Gelfand-Tsetlin tableau.
 \begin{itemize}
\item[(i)]
We say that $T(L)$ satisfies a relation $(i,j)\geq (r,s)$ (respectively, $(i,j)> (r,s)$) if $l_{ij}-l_{st}\in \mathbb{Z}_{\geq 0}+\frac{1(q)}{2}\ (\text{respectively, }l_{ij}-l_{st}\in \mathbb{Z}_{> 0}+\frac{1(q)}{2})$.
\item[(ii)]
We say that a Gelfand-Tsetlin tableau $T(L)$ \emph{satisfies $\mathcal{C}$} if $T(L)$ satisfies all the relations in $\mathcal{C}$
and $l_{ki}-l_{kj}\in \mathbb{Z}+\frac{1(q)}{2}$ only if $(k,i)$ and $(k,j)$
in the same indecomposable subset of $\mathfrak{V}(\mathcal{C})$. In this case we call $T(L)$ a $\mathcal{C}$-\emph{\bf realization}.
\item[(iii)] $\mathcal{C}$ is a \emph{\bf maximal} set of relations for $T(L)$ if $T(L)$ satisfies $\mathcal{C}$ and whenever $T(L)$ satisfies a set of relations $\mathcal{C}'$ we have that $\mathcal{C}$ implies $\mathcal{C}'$.
\item[(iv)] If $T(L)$  satisfies  $\mathcal{C}$ we denote by ${\mathcal B}_{\mathcal{C}}(T(L))$  the subset of ${\mathcal B}(T(L))$ of all Gelfand-Tsetlin tableaux satisfying $\mathcal{C}$, and by $V_{\mathcal{C}}(T(L))$ the complex vector space spanned by ${\mathcal B}_{\mathcal{C}}(T(L))$.
\end{itemize}
\end{definition}

Set
\begin{equation}\label{c-for1}
e_{ki}(L)=\left\{
\begin{array}{cc}
0,& \text{ if } T(L)\notin  \mathcal {B}_{\mathcal{C}}(T(L))\\
-\frac{\prod_{j=1}^{k+1}[l_{ki}-l_{k+1,j}]_q}{\prod_{j\neq i}^{k}[l_{ki}-l_{kj}]_q},& \text{ if } T(L)\in  \mathcal {B}_{\mathcal{C}}(T(L))
\end{array}
\right.
\end{equation}

\begin{equation}\label{c-for2}
f_{ki}(L)=\left\{
\begin{array}{cc}
0,& \text{ if } T(L)\notin  \mathcal {B}_{\mathcal{C}}(T(L))\\
\frac{\prod_{j=1}^{k-1}[l_{ki}-l_{k-1,j}]_q}{\prod_{j\neq i}^{k}[l_{ki}-l_{kj}]_q},& \text{ if } T(L)\in  \mathcal {B}_{\mathcal{C}}(T(L))
\end{array}
\right.
\end{equation}

\begin{equation}\label{c-for3}
h_{k}(L)=\left\{
\begin{array}{cc}
0,& \text{ if } T(L)\notin  \mathcal {B}_{\mathcal{C}}(T(L))\\
q^{2\sum_{i=1}^{k}l_{ki}-\sum_{i=1}^{k-1}l_{k-1,i}-\sum_{i=1}^{k+1}l_{k+1,i} -1},& \text{ if } T(L)\in\mathcal {B}_{\mathcal{C}}(T(L))
\end{array}
\right.
\end{equation}

%\begin{equation}
%\Phi(L,z_1,\ldots,z_m)=
%\left\{
%\begin{array}{cc}
%1,& \text{ if }T(L+z_1+\ldots+z_t)\in  \mathcal {B}_{\mathcal{C}}(T(L)) \text{ for any } t\\
%0,& \text{ otherwise}.
%\end{array}
%\right.
%\end{equation}

%In \cite{FRZ2} it is  shown that any admissible set of relations $\mathcal{C}$ leads to a family of $U_{q}$-modules. In fact, for any $\mathcal{C}$-realization $T(L)$ the space  $V_{\mathcal{C}}(T(L))$ %has a Gelfand-Tsetlin module structure with the action of the generators of $U_{q}$ given by the Gelfand-Tsetlin formulas  (\ref{Gelfand-Tsetlin formulas}) with coefficients given by (\ref{c-for1}), %(\ref{c-for2}) and (\ref{c-for3}).

\begin{definition}
Let $\mathcal{C}$ be a subset of $\mathcal{R}$. We call $\mathcal{C}$  \emph{\bf realizable} if for any tableau $T(L)$ satisfying  $\mathcal{C}$,  the vector space $V_{\mathcal{C}}(T(L))$  has a structure of a $U_q$-module, endowed with the action of $U_q$ given by the Gelfand-Tsetlin formulas (\ref{Gelfand-Tsetlin formulas}) with coefficients given by (\ref{c-for1}), (\ref{c-for2}) and (\ref{c-for3}).
\end{definition}

%\begin{lemma}[\cite{FRZ}, Proposition 4.18 and Lemma 4.19]\label{behavior on boundary}
%Let $\mathcal{C}$ be an admissible set. If $T(L)$ is a $\mathcal{C}$-realization such that $l_{ki}-l_{kj}\in 1+\frac{1(q)}{2}$ then one of the following cases hold.
%\begin{itemize}
%\item[(i)] $\{(k,i)\geq (k-1,s)>(k,j)\}\subseteq \mathcal{C}$.
%\item[(ii)] $\{(k,i)>(k+1,s)\geq (k,j)\}\subseteq \mathcal{C}$.
%\item[(iii)] $(n-1,i)>(n,s)\in \mathcal{C}$ and $(n,t)\geq (n-1,j)\in \mathcal{C}$, for some $s\leq t$.
%\end{itemize}
%Moreover, $$\# \left\{l_{k+1,i'},\ l_{k-1,j'}\ |\ l_{k+1,i'}-l_{kj}\in\frac{1(q)}{2},\ l_{k-1,j'}-l_{ki}\in\frac{1(q)}{2}\right\}\geq2.$$
%\end{lemma}
%\begin{proof}
%{\color{red} The proof is the same as in the nonquantum case.}
%\end{proof}

\begin{theorem}\label{sufficiency of admissible}[\cite{FRZ2} Theorem 3.9 and Theorem 4.1] If $\mathcal{C}$ is a union of disconnected sets from $\mathfrak{F}$ then $\mathcal{C}$ is realizable. Moreover, $V_{\mathcal{C}}(T(L))$ is a Gelfand-Tsetlin module with the generator $c_{mk}$ of $\Gamma_q$ acting on $T(L)\in\mathcal{B}(T(v))$ as multiplication by
\begin{equation}\label{eigenvalues of gamma_mk}
\gamma_{mk}(L)=(k)_{q^{-2}}!(m-k)_{q^{-2}}!q^{k(k+1)+\frac{m(m-3)}{2}}\sum_{\tau}
q^{\sum_{i=1}^{k}l_{m\tau(i)}-\sum_{i=k+1}^{m}l_{m\tau(i)}}
\end{equation}
where $\tau\in S_m$ is such that $\tau(1)<\cdots<\tau(k), \tau(k+1)<\cdots<\tau(m)$.

\end{theorem}

\begin{remark}
The sets $\mathcal{S} :=\{(i+1,j)\geq(i,j)>(i+1,j+1)\ |\ 1\leq j\leq i\leq n-1\}$ and $\emptyset$ are indecomposable sets of relations. By Theorem 2.11 in \cite{UST2}, any irreducible finite dimensional $U_{q}$-module is isomorphic to $V_{\mathcal{S}}(T(v))$ for some $v$. The family of modules $\{V_{\mathcal{\emptyset}}(T(v))\}$ coincide with the family of generic modules as constructed in \cite{MT}, Theorem 2.
\end{remark}

\section{New Gelfand-Tsetlin modules} \label{sec-der}
In \cite{FGR3} the classical Gelfand-Tsetlin formulas were generalized allowing to construct a new family of $\mathfrak{gl}_n$-modules associated with tableaux with at most one singular pair. In this section we will use this approach and combine with the ideas of Section \ref{section modules with relations}  to construct new families of $U_{q}$-modules.

\begin{definition}
A Gelfand-Tsetlin tableu $T(v)$ will be called {\bf $\mathcal{C}$-singular} if:
\begin{itemize}
\item[(i)] $T(v)$ satisfies $\mathcal{C}$.
\item[(ii)] There exists $1\leq s<t\leq r\leq n-1$ such that $\mathfrak{V}(\mathcal{C})\cap\{(r,s),(r,t)\}=\emptyset$ and  $v_{rs}-v_{rt}\in\frac{1(q)}{2}+\mathbb{Z}$.
\end{itemize}
The tableau $T(v)$ will be called  {\bf ($1$, $\mathcal{C}$)-singular} if it is \emph{$\mathcal{C}$-singular} and the tuple $(r,s,t)$ is unique. The tableau will be called {\bf $\mathcal{C}$-generic} if it is not $\mathcal{C}$-singular.
\end{definition}
From now on we will fix an admissible set of relations $\mathcal{C}$, $T(\bar{v})$ a $(1,\mathcal{C})$-singular tableau, $(i,j,k)$ such that $\bar{v}_{ki}=\bar{v}_{kj}$ with $1\leq i < j \leq k\leq n-1$ and by $\tau$ we denote  the element in $S_{n-1} \times\cdots \times S_{1}$ such that $\tau[k]$ is the transposition $(i,j)$  and all other $\tau[t]$ are $\mbox{Id}$.

By ${\mathcal H}$ we denote the hyperplanes $v_{ki} - v_{kj}\in  \frac{1(q)}{2}+\mathbb{Z}$ in ${\mathbb C}^{\frac{n(n-1)}{2}}$ and let $\overline{\mathcal H}$ be the subset of all $v$ in ${\mathbb C}^{\frac{n(n-1)}{2}}$ such that $v_{tr} \neq v_{ts}$ for all triples $(t,r,s)$ except for $(t,r,s) = (k,i,j)$. By ${\mathcal F}_{ij}$  the space of all rational functions that are smooth on $\overline{\mathcal H}$.

We impose the conditions  $T(\bar{v} + z) = T(\bar{v} +\tau(z))$ on the corresponding tableaux and formally introduce  new tableaux  ${\mathcal D} T({\bar{v}} + z)$ for  every $z \in {\mathbb Z}^{\frac{n(n-1)}{2}}$ subject to the relations ${\mathcal D} T({\bar{v}} + z) + {\mathcal D} T({\bar{v}} + \tau(z)) = 0$. We call ${\mathcal D} T(u)$  {\it the derivative Gelfand-Tsetlin tableau} associated with $u$.

\begin{definition}
We set  $V_{\mathcal{C}}(T(\bar{v}))$ to be the $\mathbb{C}$-vector space spanned by the set of tableaux $\{ T(\bar{v} + z), \,\mathcal{D} T(\bar{v} + z) \; | \; T(\bar{v} + z)\in\mathcal{B}_{\mathcal{C}}(T(\bar{v}))\}.$ A basis of $V_{\mathcal{C}}(T(\bar{v}))$ is for example the set
$$
\{ T(\bar{v} + z), \mathcal{D} T(\bar{v} + z') \; | \; T(\bar{v} + z),\ T(\bar{v} + z')\in\mathcal{B}_{\mathcal{C}}(T(\bar{v})), \text{ and } z_{ki} \leq z_{kj}, z'_{ki} > z'_{kj}\}.
$$
\end{definition}

In order to define the action of the generators of $U_q$ on $V_{\mathcal{C}}T(\bar{v})$ we first note that for any $\mathcal{C}$-generic vector $v'$ such that $\mathcal{C}$ is a maximal set of relations for $T(v')$, $V_{\mathcal{C}}(T(v'))$ has a structure of  an irreducible $U_{q}$-module. If $v$ denotes the vector with entries $v_{rs}=v'_{rs}$ if $(r,s)\neq (k,i), (k,j)$ and $v_{ki}=x$, $v_{kj}=y$ variables, then $V_{\mathcal{C}}(T(v))$ is a $U_{q}$-module over $\mathbb{C}(x,y)$ with action of the generators given by the formulas (\ref{Gelfand-Tsetlin formulas}). From now on, by $v$ we will denote one such vector.

Define linear map  $\mathcal{D}^{\bar{v}}:  {\mathcal F}_{ij} \otimes V_{\mathcal{C}}(T(v)) \to  V_{\mathcal{C}}(T(\bar{v}))   $ by
$$
\mathcal{D}^{\bar{v}} (f T(v+z)) = \mathcal{D}^{\bar{v}} (f) T(\bar{v}+z) +   f(\bar{v}) \mathcal{D} T(\bar{v}+z),
$$
where $\mathcal{D}^{\bar{v}}(f) = \frac{q-q^{-1}}{4\ln q}\left(\frac{\partial f}{\partial x}-\frac{\partial f}{\partial y}\right)(\bar{v})$. In particular,
$\mathcal{D}^{\bar{v}} ([x-y]_qT(v+z)) = T(\bar{v}+z)$,
 $\mathcal{D}^{\bar{v}} (T(v+z)) = \mathcal{D} (T(\bar{v}+z))$.

\subsection{Module structure on $V_{\mathcal{C}}(T(\bar{v}))$}

We define the action of $e_r,f_r,q^{\epsilon_r}(r=1,2,\ldots,n)$ on the generators of $V_{\mathcal{C}}(T(\bar{v}))$ as follows:
\begin{align*}
g(T(\bar{v} + z))=&\  \mathcal{D}^{\bar{v}}([x - y]_qg(T(v + z)))\\
g(\mathcal{D}T(\bar{v} + z')))=&\ \mathcal{D}^{\bar{v}} ( g(T(v + z'))),
\end{align*}
where $z, z' \in {\mathbb Z}^{\frac{n(n-1)}{2}}$ with $z' \neq \tau(z')$,
$g\in \{e_r,f_r,q^{\epsilon_r}(r=1,2,\ldots,n)\}$. One should note that $[x - y]_qg(T(v + z))$ and $g(T(v + z'))$ are in $ {\mathcal F}_{ij} \otimes V_{\mathcal{C}}(T(v))$, so the right hand sides in the above formulas are well defined.

\begin{proposition} \label{compatible} Let $z \in {\mathbb Z}^{\frac{n(n-1)}{2}}$ and
$g\in \{e_r,f_r,q^{\epsilon_r}(r=1,2,\ldots,n)\}$.
\begin{itemize}
\item[(i)] $ {\mathcal D}^{\bar{v}} ([x-y]_q g T(v+z)) =  {\mathcal D}^{\bar{v}} ([x-y]_q g T(v+\tau(z))) $ for all $z$.
\item[(ii)] $ {\mathcal D}^{\ov}(g T(v+z)) =  - {\mathcal D}^{\bar{v}} (g T(v+\tau(z))) $ for all $z$ such that $\tau(z) \neq z$.
\end{itemize}
\end{proposition}

It remains to prove that this well-defined action endows $V_{\mathcal{C}}(T(\bar{v}))$  with a $U_q$-module structure.

\begin{lemma} \label{dij-commute}
Let $g$ be generator of $U_q$.
\begin{itemize}
\item[(i)] $g (T(\ov + z)) = {\rm ev} ( \ov) g ( T(v+z))$ whenever $\tau(z) \neq z$.
\item[(ii)] $\mathcal{D}^{\ov} g (F) = g \mathcal{D}^{\ov} (F)$ if $F$ and $g(F)$ are in ${\mathcal F}_{ij} \otimes {V}(T(v))$.
\item[(iii)] $\mathcal{D}^{\ov} ( [x-y]_qg (F)) = g( {\rm ev} (\ov) F)$ if $F \in {\mathcal F}_{ij} \otimes {V}(T(v))$.
\end{itemize}
\end{lemma}
\begin{proof}

(i) Since $\tau(z) \neq z$, $ g ( T(v+z))\in {\mathcal F}_{ij} \otimes {V}(T(v))$. Thus
\begin{align*}
g (T(\ov + z))&=\mathcal{D}^{\ov} ([x-y]_q g (T(v+ z)))\\
&=\mathcal{D}^{\ov} ([x-y]_q)\mbox{ev}(\ov)g (T(v+ z))
+\mbox{ev}(\ov)([x-y]_q) \mathcal{D}^{\ov}g (T(v+ z)\\
&=\mbox{ev}(\ov)g (T(v+ z).
\end{align*}
(ii) Using (i) and the facts that $g( T(v+z)) $ is in  ${\mathcal F}_{ij} \otimes {\mathcal V}$ and $g(\mathcal{D}T(\bar{v} + z))=\mathcal{D}^{\ov}(g( T(v + z)))$ we have
$$
\mathcal{D}^{\ov}g ( f T(v+z))-g \mathcal{D}^{\ov} ( f T(v+z)) = \mathcal{D}^{\ov} (f) \left( \mbox{ev} (\ov)g( T(v+z))  - g( T(\ov+z))  \right) =0.
$$

(iii) Taking into consideration that $ [x-y]_q g(F)\in  {\mathcal F}_{ij} \otimes {V}(T(v))$, by (ii) we have
\begin{eqnarray*}
 \mathcal{D}^{\ov} \left( [x-y]_q g (fT(v + z)) \right) & = & g \mathcal{D}^{\ov} \left( [x-y]_qf(T(v + z)) \right)\\
& = & g ( \mbox{ev}(\ov) f T(v+z)).
\end{eqnarray*}

\end{proof}

\begin{proposition} \label{t-v-rep}
Set $\bar{v}$ be any the fixed $(1,\mathcal{C})$-singular vector and $z\in\mathbb{Z}^{\frac{n(n-1)}{2}}$.
\begin{itemize}
\item[(i)]
$q^{0}(T(\ov + z))=T(\ov + z),\  q^{h}q^{h'}(T(\ov + z))=q^{h+h'}(T(\ov + z)) \quad (h,h' \in  P)$.
\item[(ii)]$q^{h}e_{r}q^{-h}T(\ov + z)=q^{\langle h,\alpha_r\rangle}e_{r} (T(\ov + z))$.
\item[(iii)]$q^{h}f_{r}q^{-h}(T(\ov + z))=q^{-\langle h,\alpha_r\rangle}f_{r}(T(\ov + z))$.
\item[(iv)]$(e_{r}f_{s}-f_{s}e_{r})(T(\ov + z))=\delta_{rs}\frac{q^{\alpha_r}-q^{-\alpha_r}}{q-q^{-1}}(T(\ov + z))$.
\item[(v)]$(e_{r}^2e_{s}-(q+q^{-1})e_re_se_r+e_je_{r}^2)  (T(\ov + z))=0  \quad (|r-s|=1)$.
\item[(vi)]$(f_{r}^2f_{s}-(q+q^{-1})f_rf_sf_r+f_sf_{r}^2)  (T(\ov + z))=0  \quad (|r-s|=1)$.
\item[(vii)]
$e_{r}e_s (T(\ov + z))=e_se_r (T(\ov + z))$, and $
  f_rf_s (T(\ov + z))=f_sf_r (T(\ov + z))$, $(|r-s|>1)$.
\end{itemize}
\end{proposition}
\begin{proof}
We only give the proof of (v). Other statements can be proved similarly.
\begin{align*}
&(e_{r}^2e_{s}-(q+q^{-1})e_re_se_r+e_se_{r}^2)  (T(\ov + z))\\
=&(e_{r}^2e_{s}-(q+q^{-1})e_re_se_r+e_se_{s}^2) \Dv ([x-y]_qT(v + z))
\end{align*}
For any $r_1,r_2,r_3$, if $\#\{r_t:r_t=k\}\leq 2$, then $[x-y]_qT(v + z)$,  $[x-y]_q e_{r_1}T(v + z)$, $[x-y]_qe_{r_2}e_{r_1}T(v + z)$, $[x-y]_qe_{r_3}e_{r_2}e_{r_1}T(v + z)$ are in
${\mathcal F}_{ij} \otimes {V}(T(v))$.
It follows from Lemma \ref{dij-commute} (ii) that
\begin{align*}
&(e_{i}^2e_{j}-(q+q^{-1})e_ie_je_i+e_je_{i}^2)  (T(\ov + z))\\
=& \Dv((e_{i}^2e_{j}-(q+q^{-1})e_ie_je_i+e_je_{i}^2) ([x-y]_qT(v + z)))\\
=&0.
\end{align*}
\end{proof}

The proof of the following proposition will be given in Appendix \ref{Section: Appendix}.
\begin{proposition} \label{d-t-v-rep}
Set $\bar{v}$ be any the fixed $(1,\mathcal{C})$-singular vector and $z\in\mathbb{Z}^{\frac{n(n-1)}{2}}$ such that $\tau(z)\neq z$. Then
\begin{itemize}
\item[(i)]$
q^{0}(\D T(\ov + z))=(\D T(\ov+ z))$, and $q^{h}q^{h'}(\D T(\ov + z))=q^{h+h'}(T(\ov + z))$, for any $h,h' \in  P$.
\item[(ii)]$q^{h}e_{r}q^{-h}(\D T(\ov + z))=q^{\langle h,\alpha_r\rangle}e_{r} (\D T(\ov + z))$.
\item[(iii)]$q^{h}f_{r}q^{-h}(\D T(\ov + z))=q^{-\langle h,\alpha_r\rangle}f_{r}(\D T(\ov + z))$.
\item[(iv)]$(e_{r}f_{s}-f_{s}e_{r})(\D T(\ov + z))=\delta_{rs}\frac{q^{\alpha_r}-q^{-\alpha_r}}{q-q^{-1}}(\D T(\ov + z))$.
\item[(v)]$(e_{r}^2e_{s}-(q+q^{-1})e_re_se_r+e_se_{r}^2)  (\D T(\ov + z))=0  \quad (|r-s|=1)$.
\item[(vi)]$(f_{r}^2f_{s}-(q+q^{-1})f_rf_sf_r+f_sf_{r}^2)  (\D T(\ov + z))=0  \quad (|r-s|=1)$.
\item[(vii)]
$e_{r}e_s (\D T(\ov + z))=e_se_r (\D T(\ov + z))$, and $
  f_rf_s (\D T(\ov + z))=f_sf_r (\D T(\ov + z))$, $(|r-s|>1)$.
\end{itemize}
\end{proposition}

Combining Propositions \ref{t-v-rep} and \ref{d-t-v-rep} we obtain our main result.

\begin{theorem}\label{thm-main}
If $\bar{v}$ is an $(1,\mathcal{C})$-singular vector in ${\mathbb C}^{\frac{n(n+1)}{2}}$, then  $V_{\mathcal{C}}(T(\bar{v}))$ is a $U_q$-module, with action of the generators of $U_q$ given by
\begin{align*}
g(T(\bar{v} + z))=&\  \mathcal{D}^{\bar{v}}([x-y]_qg(T(v + z)))\\
g(\mathcal{D}T(\bar{v} + z')))=&\ \mathcal{D}^{\bar{v}} ( g(T(v + z'))),
\end{align*}
for any $z,z'\in\mathbb{Z}^{\frac{n(n-1)}{2}}$ with $z'\neq\tau(z')$.
 \end{theorem}

\subsection{Action of the  Gelfand-Tsetlin subalgebra on  $V_{\mathcal{C}}(T(\bar{v}))$}

\begin{lemma}\label{generators of Gamma acting in many tableaux}
\begin{itemize}
\item[(i)] If $z\in\mathbb{Z}^{\frac{n(n-1)}{2}}$ is such that $|z_{ki}-z_{kj}|\geq n-m$ for some $0 \leq m \leq n$, then for each $1\leq s \leq r\leq n-m$ we have:
\begin{itemize}
\item[(a)]
$
c_{rs}(T(\bar{v}+z))=\Dv([x-y]_qc_{rs}(T(v+z))),
$
\item[(b)]
$c_{rs}(\mathcal{D}T(\bar{v}+z))= \Dv(c_{rs}(T(v+z)))$ if $z\neq \tau(z)$,
\end{itemize}
\item[(ii)] If $1\leq s\leq r\leq k$ and $z\in\mathbb{Z}^{\frac{n(n-1)}{2}}$ then the action of $c_{rs}$ on $T(\bar{v}+z)$ and
$\mathcal{D}T(\bar{v}+z)$ is defined by the formulas in {\rm (i)}.

\end{itemize}
\end{lemma}
\begin{proof}
Recall that
$$c_{rs}=\sum_{\sigma,\sigma'\in S_r}(-q)^{l(\sigma)+l(\sigma')}
l_{\sigma(1),\sigma'(1)}^{+}\cdots l_{\sigma(s),\sigma'(s)}^{+}l_{\sigma(s+1),\sigma'(s+1)}^{-}\cdots l_{\sigma(r),\sigma'(r)}^{-}.$$
\begin{itemize}
\item[(i)]
The elements $l_{rs}^+$ with $r<s$ can be written as sum of products of elements of the form
$l_{r,r+1}^+, l_{r+1,r+2}^+, \ldots, l_{s-1,s}^+$ and $l_{sr}^-$, $r<s$  can be written as sum of products of  $l_{s,s-1}^-, l_{s-1,s-2}^-  \ldots, l_{r+1,r}^-$.
 By the hypothesis $|z_{ki}-z_{kj}|\geq n-m$, the coefficients  that appear in the decompositions of the following vectors are all in $\mathcal{F}_{ij}$:
 $[x-y]_q l_{\sigma(t),\sigma'(t)}^{-}\cdots l_{\sigma(r),\sigma'(r)}^{-}(T(v+z))(t>s)$,
  $[x-y]_q l_{\sigma(t),\sigma'(t)}^{+}\cdots l_{\sigma(r),\sigma'(r)}^{-}(T(v+z))(t\leq s)$,
   $ l_{\sigma(t),\sigma'(t)}^{-}\cdots l_{\sigma(r),\sigma'(r)}^{-}(T(v+z))$ ($t>s,z\neq\tau(z)$),
  $ l_{\sigma(t),\sigma'(t)}^{+}\cdots l_{\sigma(r),\sigma'(r)}^{-}(T(v+z))$ ($t\leq s,z\neq\tau(z)$)
 Then the statement follows from Lemma \ref{dij-commute}(ii).
\item[(ii)]   As $1\leq s\leq r\leq k$, then the every tableau that appears in the following elements have the same $(k,i)$th and $(k,j)$th entries:
\begin{align*}
[x-y]_q l_{\sigma(t),\sigma'(t)}^{-}\cdots l_{\sigma(r),\sigma'(r)}^{-}(T(v+z))& &(t>s),\\
[x-y]_q l_{\sigma(t),\sigma'(t)}^{+}\cdots l_{\sigma(r),\sigma'(r)}^{-}(T(v+z))& &(t\leq s),\\
 l_{\sigma(t),\sigma'(t)}^{-}\cdots l_{\sigma(r),\sigma'(r)}^{-}(T(v+z))& &(t>s,z\neq\tau(z)),\\
l_{\sigma(t),\sigma'(t)}^{+}\cdots l_{\sigma(r),\sigma'(r)}^{-}(T(v+z))& &(t\leq s,z\neq\tau(z)).
\end{align*} Hence all of the listed vectors are in ${\mathcal F}_{ij} \otimes {V}(T(v))$ and  Lemma \ref{dij-commute}(ii)  completes the proof.
\end{itemize}
    \end{proof}

\begin{lemma}\label{lem-ck2}
Assume $\tau(z)\neq z$. Then we have:
\begin{itemize}
\item[(i)] $c_{kr}(T(\bar{v}+z))=\gamma_{kr}(\bar{v}+z)T(\bar{v}+z)$
\item[(ii)] $(c_{kr}-\gamma_{kr})(\mathcal{D}T(\bar{v}+z))=0$ if and only if $r\in\{0,k\}$.
\item[(iii)] $(c_{kr}-\gamma_{kr}(\bar{v}+z))^{2}\mathcal{D}T(\bar{v}+z)=0.$
\end{itemize}

\end{lemma}

\begin{proof}
Recall that if $w$ is a $\mathcal{C}$-generic vector then $ c_{kr}T(w)=\gamma_{kr}(w)T(w)$, where $\gamma_{kr}(w)$ is a symmetric function in variables $w_{k1},\ldots,w_{kk}$.
\begin{itemize}
\item[(i)]  By Lemma \ref{generators of Gamma acting in many tableaux}(ii),  we have
\begin{align*}
c_{kr}(T(\bar{v}+z))=&\Dv([x-y]_qc_{kr}(T(v+z)))\\
=&\Dv([x-y]_q\gamma_{kr}(v+z)T(v+z))\\
=&\gamma_{kr}(\bar{v}+z)T(\bar{v}+z).
\end{align*}
\item[(ii)] Also by Lemma \ref{generators of Gamma acting in many tableaux}(ii) we have:
\begin{align*}
c_{kr}(\mathcal{D}T(\bar{v}+z)))=& \Dv(c_{kr}(T(v+z)))\\
=& \Dv(\gamma_{kr}(v+z)T(v+z))\\
=& \Dv(\gamma_{kr}(v+z))T(\bar{v}+z)+\gamma_{kr}(\bar{v}+z)\mathcal{D}T(\bar{v}+z)
\end{align*}
When $r=0, k$, $\gamma_{kr}(v+z)$ is symmetric function in variables $v_{ki}, v_{kj}$.
Then $\Dv(\gamma_{kr}(v+z))=0$, one has that $c_{kr}(T(\bar{v}+z))=\gamma_{kr}(\bar{v}+z)T(\bar{v}+z)$.
When $1\leq r \leq k-1$, $\gamma_{kr}(v+z)$ is not symmetric.
$\Dv(\gamma_{kr}(v+z))=\frac{a(q-q^{-1})^2[z_{ki}-z_{kj}]_q }{2}\neq 0$ where $a$ is the coefficient of $q^{(v_{ki}+z_{ki})-(v_{kj}+z_{kj})}$ in $\gamma_{kr}(v+z)$.
\item[(iii)]  This part follows from (i) and (ii).
\end{itemize}
\end{proof}

The following statement follows directly from the action of generators of $U_q$ and the Gelfand-Tsetlin subalgebra.

\begin{lemma}\label{Gamma k separates tableaux}
Let $\Gamma_{k-1}$ be the subalgebra of $\Gamma$ generated by $\{c_{rs}:1\leq s\leq r\leq k-1\}$.
 For any $m\in\mathbb{Z}_{\geq 0}$ let $R_{m}$ be the set of $z\in\mathbb{Z}^{\frac{n(n-1)}{2}}$ such that $|z_{ki}-z_{kj}|=m$.
\begin{itemize}
\item[(i)]
 If $z,z'\in\mathbb{Z}^{\frac{n(n-1)}{2}}$ are such that $z_{rs}\neq z'_{rs}$ for some $1\leq s\leq r\leq k-1$. Then, $\Gamma_{k-1}$ separates the tableaux $T(\bar{v}+z)$ and $T(\bar{v}+z')$, that is, there exist $c\in\Gamma_{k-1}$ and $\gamma\in\mathbb{C}$ such that $(c-\gamma)T(\bar{v}+z)=0$ but $(c-\gamma)T(\bar{v}+z')\neq 0$.
\item[(ii)]
 If $z\in R_{m}$ then  there exists $\bar{z}\in R_{m+1}$ such that $T(\bar{v}+z)$ appears with non-zero coefficient in the decomposition of $f_{k}f_{k-1}\cdots f_{k-t}T(\bar{v}+\bar{z})$ for some $t\in\{0,1,\ldots,k-1\}$.
\end{itemize}
\end{lemma}

Now we can prove our second main result:

 \begin{theorem}\label{GT module structure}
The module $V_{\mathcal{C}}(T(\bar{v}))$ is an $(1,\mathcal{C})$-singular Gelfand-Tsetlin module. Moreover for any $z\in\mathbb{Z}^{\frac{n(n-1)}{2}}$ and any $1\leq r \leq s\leq n$ the following identities hold.
\begin{equation}\label{Gamma acting on T}
c_{rs}(T(\bar{v}+z))=\Dv([x-y]_qc_{rs}(T(v+z)))
\end{equation}
\begin{equation}\label{Gamma acting on DT}
c_{rs}(\mathcal{D}T(\bar{v}+z))= \Dv(c_{rs}(T(v+z))), \text{ if } z\neq \tau(z).
\end{equation}
\end{theorem}

\begin{proof}

Let $R_{\geq n} := \cup_{m\geq n} R_m$. For any $z\in R_{\geq n}$ consider the submodule $W_{z}$ of $V_{\mathcal{C}}(T(\bar{v}))$ generated by $T(\bar{v}+z)$. By Lemma \ref{generators of Gamma acting in many tableaux}(i)(a), $T(\bar{v}+z)$ is a common eigenvector of all generators of $\Gamma$ and thus $W_{z}$ is a Gelfand-Tsetlin module by Lemma \ref{lem-cyclic-Gelfand-Tsetlin}. Then $W=\sum_{z\in R_{\geq n}}W_{z}$  is also a Gelfand-Tsetlin module. We first show that $W$ contains all tableau $T(\bar{v}+z)$ for any $z\in\mathbb{Z}^{\frac{n(n-1)}{2}}$.
   Indeed, assume that $|z_{ki}-z_{kj}|= n-1$ and consider $T(\bar{v}+z)$. Then, by Lemma \ref{Gamma k separates tableaux} there exists $z'\in R_n$ and a nonzero $x\in U_{q}(\gl_{k+1})$ such that  $xT(\bar{v}+z')=T(\bar{v}+z)$.
%Following Remark \ref{elements in Rm}, we assume that $N=1$ and $z^{(0)} , z^{(1)}\in R_{n-1}$ without loss of generality since $z^{(m)}$ in $R_{\geq n}$ implies $T(\bar{v}+z^{(m)})\in W$. The action of  all generators $\{c_{rs}\}_{1\leq r \leq s\leq n-1}$ of $\Gamma$, except for the center of $U_q$, on  $T(\bar{v}+z)$ and $T(\bar{v}+z^{(1)})$ is determined by (\ref{Gamma acting on T}).
Let $c\in \Gamma$ be a central element and $(c-\gamma)T(\bar{v}+z')=0$ for some complex $\gamma$. Then $(c-\gamma)xT(\bar{v}+z')=0$ $=(c-\gamma) (T(\bar{v}+z))$.
  %  Recall that by Lemma \ref{Gamma k separates tableaux}, there exists $C\in \Gamma_{k-1}$ which  acts with different scalars on $T(\bar{v}+z)$ and $T(\bar{v}+z^{(1)})$. Since $C$ commutes with $(c-\gamma)$, both $T(\bar{v}+z)$ and $T(\bar{v}+z^{(1)})$ are in $W$. Moreover,
%  $(c-\gamma)T(\bar{v}+z)=$ $(c-\gamma)T(\bar{v}+z^{(1)})=0$. Hence, the action of $\Gamma$ on any $T(\bar{v}+z)$ with $z\in R_{n-1}$ is given by (\ref{Gamma acting on T}).
%   Moreover, $T(\bar{v}+z)\in W$ for any $z\in R_{n-1}$.
%  Next we consider a tableau $T(\bar{v}+z)$ with $z\in R_{n-2}$. Again
%  by Lemma~\ref{lem-connect} one finds a nonzero $y\in \gl_{k+1}$ and
%  $z'\in R_{n-1}$ such that  $y T(\bar{v}+z')$ contains $T(\bar{v}+z)$ and at most one other tableau. For all generators of centers of $U(\gl_m)$, $m\leq n-2$ the statement follows from Lemma \ref{generators of Gamma acting in many tableaux}.
%  If $c$ is in the center of $U$ or in the center of $U(\gl_{n-1})$ then it commutes with $y$.  Choose $C\in \Gamma_{k-1}$ which separates the tableaux in the image $y T(\bar{v}+z')$ and which acts by a scalar on the tableau
%  $T(\bar{v}+z')$. Applying the argument above we conclude that the action of $\Gamma$ on any $T(\bar{v}+z)$ with $z\in R_{n-2}$ is determined  by (\ref{Gamma acting on T}) and $T(\bar{v}+z)\in W$ for any $z\in R_{n-2}$.
  Continuing analogously with the sets
  $R_{n-3}, \ldots, R_0$ we show that any tableau $T(\bar{v}+z)$  belongs to $W$.
%
%   Note that when $z\in R_0$, $\tau(z)=z$. In this case it will be the unique such tableau coming from some  $T(\bar{v}+z')$ with $z'\in R_1$ and ``separation'' is not needed.

  Consider the  quotient $\overline{W}=V(T(\bar{v}))/W$.
 The vector $\mathcal{D}T(\bar{v}+z)+W$ of $\overline{W}$ is a common eigenvector of $\Gamma$ by Lemma \ref{generators of Gamma acting in many tableaux}(i)(b) for any $z\in R_n$. We can repeat now the argument above substituting everywhere the tableaux
  $T(\bar{v}+z)$ by $\mathcal{D}T(\bar{v}+z)$. Hence, $\overline{W}=\sum_{z\in R_n}\overline{W}_{z}$, where
  $\overline{W}_{z}$ denotes the submodule of $\overline{W}$ generated by  $\mathcal{D}T(\bar{v}+z) +W$.
  By Lemma \ref{lem-cyclic-Gelfand-Tsetlin} we conclude that $\overline{W}$ is a Gelfand-Tsetlin module. Therefore,
  $V_{\mathcal{C}}(T(\bar{v}))$ is a Gelfand-Tsetlin module with action of $\Gamma$ given by (\ref{Gamma acting on T}) and (\ref{Gamma acting on DT}).
\end{proof}

As a consequence of Theorem \ref{GT module structure} we have $\dim V_{\mathcal{C}}(T(\bar{v}))(\sm)\leq 2$ for any $\sm\in\Sp\Ga$. Here, a  maximal ideal $\sm$ of $\Gamma_q$ is determined by the corresponding basis tableau via the formulas in Theorem \ref{GT module structure}. If the entries of this tableau which define the $1$-singular pair are distinct then the dimension of $V_{\mathcal{C}}(T(\bar{v}))(\sm)$ is exactly $2$.  Moreover, Lemma \ref{lem-ck2} implies that the generators $c_{ki}$ of $\Gamma_q$ have Jordan cells of size $2$ on $V_{\mathcal{C}}(T(\bar{v}))$.

%For a $\mathcal{C}$-generic vector $v$ and $z\in\mathbb{Z}^{\frac{n(n-1)}{2}}$, we have
%$$c_{rs}(T(v+z)) \  = \gamma_{rs}(v+z) T(v+z),$$
%where $c_{rs}$ are the generators of $\Gamma$.

Finally, we have our third main result

\begin{theorem}\label{thm-when L irred}
The module $V_{\mathcal{C}}(T(\bar{v}))$ is irreducible whenever  $\mathcal{C}$ is a maximal set of relations for $T(\bar{v})$ and $\bar{v}_{rs}-\bar{v}_{r-1,t}\notin \mathbb{Z}+\frac{1(q)}{2}$ for any $1\leq t <r\leq n$, $1\leq s\leq r$ such that $(r,s)\notin \mathfrak{V}(\mathcal{C})$ or $(r-1,t)\notin \mathfrak{V}(\mathcal{C})$ .
\end{theorem}

\begin{proof}
%We introduce the following notation which will be used in this section only.
%\begin{eqnarray*}
%L &:=& \{T(\bar{v}+z)| z\neq\tau(z)\},\\
%S& :=& \{T(\bar{v}+z)| z=\tau(z)\},\\
%D& := &\{\mathcal{D}T(\bar{v}+z)| z\neq\tau(z)\}.
%\end{eqnarray*}

Let $z\in\mathbb{Z}^{\frac{n(n-1)}{2}}$ such that $z\neq\tau(z)$ and $w=\bar{v}+z$. If $w_{rs}-w_{r-1,t}\notin \frac{1(q)}{2} +\mathbb{Z}_{\geq 0}$ for any $r,s,t$, then the module $V_{\mathcal{C}}(T(\bar{v}))$ is generated by the two tableau $T(w)$ and $\mathcal{D}T(w)$.
Let $z\in\mathbb{Z}^{\frac{n(n-1)}{2}}$ such that $z\neq\tau(z)$ and $w=\bar{v}+z$. If $w_{rs}-w_{r-1,t}\notin \frac{1(q)}{2}+\mathbb{Z}_{\geq -1}$ for any $r,s,t$, then $V_{\mathcal{C}}(T(\bar{v}))$ is generated by $\mathcal{D}T(w)$. If $z\neq\tau(z)$ and $\bar{v}_{rs}-\bar{v}_{r-1,t}\notin \frac{1(q)}{2}+\mathbb{Z}$ for any $1\leq t <r\leq n$, $1\leq s\leq r$  then $T(\bar{v}+z)$ generates $V_{\mathcal{C}}(T(\bar{v}))$.
This proves the statement.

\end{proof}

\section{New $(1,\mathcal{C})$-singular Gelfand-Tsetlin modules for $\gl_n$}

Three main results proved above, Theorem \ref{thm-main}, Theorem \ref{GT module structure} and Theorem \ref {thm-when L irred} also hold when $q=1$. We will use the same notation as before. The proofs  are analogous to the proofs of theorems above. Therefore, we obtain a new family of  $(1,\mathcal{C})$-singular Gelfand-Tsetlin modules for $\gl_n$ for any admissible set of relations $\mathcal{C}$ plus one additional singularity. These modules are irreducible for any maximal set $\mathcal{C}$.  Namely, we have

\begin{theorem}\label{thm-gl(n)}
Let $\mathcal{C}$ be an admissible set of relations and $T(\bar{v})$ a $(1,\mathcal{C})$-singular tableau. Then
\begin{itemize}
\item[(i)] $V_{\mathcal{C}}(T(\bar{v}))$ is a  $(1,\mathcal{C})$-singular Gelfand-Tsetlin $\gl_n$-module
with the action of the generators of $\mathfrak{gl}_{n}$ given by
\begin{align*}
g(T(\bar{v} + z))=&\  \mathcal{D}^{\bar{v}}((x - y)g(T(v + z)))\\
g(\mathcal{D}T(\bar{v} + z')))=&\ \mathcal{D}^{\bar{v}} (g(T(v + z'))),
\end{align*}
where $\mathcal{D}^{\bar{v}}(f) = \frac{1}{2}\left(\frac{\partial f}{\partial x}-\frac{\partial f}{\partial y}\right)(\bar{v})$. In the particular case of $\mathcal{C}=\emptyset$ we recover the $1$-singular modules constructed in \cite{FGR3}.
\item[(ii)] $V_{\mathcal{C}}(T(\bar{v}))$ is irreducible whenever $\mathcal{C}$ is the maximal set of relations satisfied by $T(\bar{v})$.
\end{itemize}
\end{theorem}

Fix $(1,\mathcal{C})$-singular tableau $T(\bar{v})$ and consider the corresponding maximal ideal $\sm=\sm_{T(\bar{v})}$ of the Gelfand-Tsetlin subalgebra $\Gamma$ \cite{FGR3}.  Since $\dim U(\gl_n)/U(\gl_n)\sm\leq 2$ there can be at most two non-isomorphic irreducible modules $V$ with $V(\sm)\neq 0$. It is not immediately clear why both of them are the subquotients of $V_{\mathcal{C}}(T(\bar{v}))$. But this fact can be shown analogously to the proof of the similar statement for the $1$-singular case in \cite{FGR5}, Theorem 5.2. So, we have

\begin{theorem}\label{thm-exhaust} Let $T(\bar{v})$ be a $(1,\mathcal{C})$-singular tableau and $\sm=\sm_{T(\bar{v})}$.
Then any  irreducible Gelfand-Tsetlin $\gl_n$-module $V$ with $V(\sm)\neq 0$  is isomorphic to a subquotients of $V_{\mathcal{C}}(T(\bar{v}))$.
\end{theorem}

Theorem \ref{thm-exhaust} completes a classification of $(1,\mathcal{C})$-singular irreducible Gelfand-Tsetlin modules for $\gl_n$.

\appendix
\section{}\label{Section: Appendix}

For a function $f = f(v)$  by $f^{\tau}$ we denote the function $f^{\tau} (v) = f (\tau (v))$.
The following lemma can be easily verified.
\begin{lemma}\label{dif operaator in functions}
Suppose $f\in {\mathcal F}_{ij}$ and $h:=\frac{f-f^{\tau}}{[x-y]_q}$.
\begin{itemize}
\item[(i)]${\mathcal D}^{\bar{v}}(f^{\tau})=-{\mathcal D}^{\bar{v}}(f)$.
\item[(ii)] If $f=f^{\tau}$, then ${\mathcal D}^{\bar{v}}(f)=0$.
\item[(iii)] If $h\in {\mathcal F}_{ij}$, then ${\rm ev} (\bar{v})(h)=2{\mathcal D}^{\bar{v}} (f)$.
\item[(iv)] ${\rm ev} (\bar{v})(f)={\mathcal D}^{\bar{v}}([x-y]_qf)$.
\item[(v)] $\mathcal{D}^{\bar{v}} (f_1f_2 T(v+z)) = \mathcal{D}^{\bar{v}} (f_1)f_2(\bar{v}) T(\bar{v}+z) +   f_1(\bar{v}) \mathcal{D} ^{\bar{v}} (f_2T(\bar{v}+z))$ for any $f_1,f_2\in {\mathcal F}_{ij}$.
\item[(vi)] $\Dv([x-y]_qf)=\Dv([x-y]_qf^{\tau})$.
\end{itemize}
\end{lemma}
\begin{proof} Items (i)-(v) follows by definition and direct computation and item (vi) follows from the following:
\begin{equation}
\Dv([x-y]_qf^{\tau})=-\Dv([y-x]_qf)=\Dv([x-y]_qf).
\end{equation}
\end{proof}
\begin{lemma}\label{sym formula}
Let $f,g$, be functions in ${\mathcal F}_{ij}$ such that $g=g^{\tau}$. Then the following identity hold.
\begin{align*}
\Dv (f) \Dv([x-y]_q g)=\Dv(fg).
\end{align*}
\end{lemma}
\begin{proof}
It is easy to be verified by definition and Lemma \ref{dif operaator in functions} (ii), (iii), (iv).
\end{proof}

\begin{lemma} \label{dv-formulas}
Let $f_m, g_m$, $m=1, \ldots, t$, be functions such that $f_m,[x-y]_qg_m$, and $\sum_{m=1}^t f_m g_m$ are in ${\mathcal F}_{ij}$ and $g_m \notin  {\mathcal F}_{ij}$. Assume also that  $\Dv (\sum_{m=1}^t f_m g_m^{\tau})=0 $ . Then the following identities hold.
\begin{itemize}
\item[(i)] $2\sum\limits_{m=1}^t{\mathcal D}^{\bar{v}} (f_m) {\mathcal D}^{\bar{v}} ( [x-y]_q g_m) =  {\mathcal D}^{\bar{v}} \left(\sum_{m=1}^t f_m g_m\right)$.
\item[(ii)]  $2\sum\limits_{m=1}^t{\mathcal D}^{\bar{v}} (f_m) {\rm ev} (\bar{v}) ( [x-y]_q g_m) =  {\rm ev} (\bar{v})  \left(\sum_{m=1}^t f_mg_m\right)$.
\end{itemize}
\end{lemma}

\begin{proof} Set for simplicity $\bar{g}_m = [x-y]_qg_m$. For (i) we use Lemma \ref{dif operaator in functions} and obtain
\begin{eqnarray*}
\left( \sum_{k=1}^t f_mg_m\right) & = & {\mathcal D}^{\bar{v}} \left(\sum_{m=1}^t f_mg_m+ \sum_{m=1}^t f_mg_m^{\tau}\right)\\
& = & \Dv \left(\sum_{m=1}^t f_m \frac{\bar{g}_m - (\bar{g}_m)^{\tau}}{[x-y]_q}
\right)\\
& = &\sum_{m=1}^t  \left( {\mathcal D}^{\bar{v}} (f_m) {\rm ev}(\bar{v}) \left( \frac{\overline{g}_m - (\overline{g}_m)^{\tau}}{[x-y]_q} \right) +   {\rm ev}(\bar{v}) (f_m)  {\mathcal D}^{\bar{v}} \left( \frac{\overline{g}_m - (\overline{g}_m)^{\tau}}{[x-y]_q} \right)\right)\\
& = &2 \sum_{m=1}^t  {\mathcal D}^{\bar{v}} (f_m) {\mathcal D}^{\bar{v}} (\overline{g}_m).
\end{eqnarray*}

For (ii) we use similar arguments.
\end{proof}

\subsection{Proof of Proposition \ref{compatible}}
For any set of relations $\mathcal{C}$ we set
\begin{equation}
\Phi(L,z_1,\ldots,z_m)=
\left\{
\begin{array}{cc}
1,& \text{ if }T(L+z_1+\ldots+z_t)\in  \mathcal {B}_{\mathcal{C}}(T(L)) \text{ for any } t,\\
0,& \text{ otherwise}.
\end{array}
\right.
\end{equation}
The action of generators can be expressed as follows:
\begin{equation}
\begin{split}
e_{k}(T(v))&=\sum_{\substack{1\leq j\leq k\\ \Phi(v,\delta^{kj})=1}}e_{kj}T(v+\delta^{kj}),\\
f_{k}(T(v))&=\sum_{\substack{1\leq j\leq k\\ \Phi(v,-\delta^{kj})=1}}f_{kj}T(v-\delta^{kj}).\\
\end{split}
\end{equation}

Let $g_{i_t}=e_{i_t}\text{ or } f_{i_t}$. For any product of generators $g_{i_1},\ldots,g_{i_r}$, the action on $T(v)$ can be expressed as
\begin{equation}\label{example phi}
\begin{split}
\sum_{\Phi(v,\pm\delta^{i_1j_1},\ldots,\pm\delta^{i_{r-1}j_{r-1}})=1}g_{i_rj_r}(v\pm\delta^{i_1j_1} \ldots \pm\delta^{i_rj_r})\cdots g_{i_1j_1}(v)
T(v\pm\delta^{i_1j_1} \ldots \pm\delta^{i_rj_r}).
\end{split}
\end{equation}
In the following all the sums satisfy the condition $\Phi=1$ as in \eqref{example phi}.

For part (i)
\begin{eqnarray*}
{\mathcal D}^{\bar{v}} ([x-y]_q e_{r} T(v+z))&= &\sum\limits_{s=1}^{r}{\mathcal D}^{\bar{v}}
([x-y]_q e_{rs}(v+z)T(v+z+\delta^{rs}))\\
&=&\sum\limits_{s=1}^{r}{\mathcal D}^{\bar{v}}([x-y]_qe_{rs}(v+z)){\rm ev} (\ov) T(v+z+\delta^{rs}))\\
& &+\sum\limits_{s=1}^{r}{\rm ev} (\ov) (([x-y]_q e_{rs}(v+z)){\mathcal D} T(T(v+z+\delta^{rs}))).
\end{eqnarray*}

The same formula holds for ${\mathcal D}^{\bar{v}} ([x-y]_q E_{rs} T(v+\tau(z)))$ after replacing $z$ with $\tau(z)$ on the right hand side. If $r\neq k$, then $e_{rs}(v+z)\in {\mathcal F}_{ij}$. Thus
\begin{eqnarray*}
&{\mathcal D}^{\bar{v}} ([x-y]_qe_{rs}(v+z)T(v+z+\delta^{rs}))={\rm ev} (\ov)(e_{rs}(v+z)T(v+z+\delta^{rs}))\\
&{\mathcal D}^{\bar{v}} ([x-y]_qe_{rs}(v+\tau(z))T(v+\tau(z)+\delta^{rs}))={\rm ev} (\ov)(e_{rs}(v+\tau(z)))T(v+\tau(z))+\delta^{rs}))
\end{eqnarray*}
Since $T(\ov+z+\delta^{rs})=T(\ov+\tau(z)+\delta^{rs})$,
${\rm ev} (\ov)(e_{rs}(v+z))={\rm ev} (\ov)(e_{rs}(v+\tau(z)))$,
${\mathcal D}^{\bar{v}} ([x-y]_q e_{r} T(v+z))={\mathcal D}^{\bar{v}} ([x-y]_q e_{r} T(v+\tau(z))) \text{ for } r\neq k$.

Suppose $r=k$, ${\mathcal D}^{\bar{v}} ([x-y]_qe_{rs}(v+z)T(v+z+\delta^{rs}))={\mathcal D}^{\bar{v}} ([x-y]_qe_{rs}(v+\tau(z))T(v+\tau(z)+\delta^{rs}))$ whenever $r\neq i,j$.
Now we consider the case when $s\in\{i,j\}$. $\tau(z+\sigma(\delta^{ki}))=\tau(z)+(\delta^{kj})$ then
$T(\ov+z+\delta^{ki})=T(\ov+\tau(z)+\delta^{kj})$,
${\mathcal D}T(\ov+z+\delta^{ki})=-{\mathcal D}T(\ov+\tau(z)+\delta^{kj})$.

The following equations follows from Lemma \ref{dif operaator in functions}.
\begin{align}
{\mathcal D}^{\bar{v}}([x-y]_qe_{ki}(v+z))={\mathcal D}^{\bar{v}}([x-y]_qe_{kj}(v+\tau(z))),  \\
{\rm ev} (\ov) ([x-y]_q e_{ki}(v+z)=-{\rm ev} (\ov) ([x-y]_q e_{kj}(v+\tau(z)).
\end{align}

The first statement for $e_r$ is proved. The proof of $f_r$ is similar. Since $q^{\varepsilon_r}T(v+z)\in {\mathcal F}_{ij}\otimes T(v+z)$, it is easy to see that statement (i) holds for $q^{\varepsilon_r}$.

The proof of part (ii) is similar.

\subsection{Proof of Proposition \ref{d-t-v-rep}}

Denote
\begin{align}
e_{kr}(L)=
-\frac{\prod_{s=1}^{k+1}[l_{kr}-l_{k+1,s}]_q}{\prod_{s\neq r}^{k}[l_{kr}-l_{ks}]_q},\\
f_{kr}(L)=
\frac{\prod_{s=1}^{k-1}[l_{kr}-l_{k-1,s}]_q}{\prod_{s\neq r}^{k}[l_{kr}-l_{ks}]_q},\\
h_{k}(L)=
q^{\sum_{r=1}^{k}l_{kr}-\sum_{r=1}^{k-1}l_{k-1,r}+k}.
\end{align}

\begin{proof}
Since $e_r T(\bar{v} + z)),f_r T(\bar{v} + z))\in  {\mathcal F}_{ij} \otimes {V}(T(v))$,
$q^h T(\bar{v} + z))\in  {\mathcal F}_{ij} \otimes T(\bar{v} + z))$. Therefore (i), (ii) and (iii) follow directly from Lemma \ref{dij-commute}.

\emph{ Proof of $(iv)$.} In the cases $r\neq k$, or $s\neq k$ or $r=s=k$, $|z_{ki}-z_{kj}|\geq 2$, the equality holds because of Lemma \ref{dij-commute}.\\
In the following we assume $r=s=k$, $|z_{ki}-z_{kj}|=1$. Without loos of generality we assume $z_{ki}=0$, $z_{kj}=1$.
\begin{align*}
&(e_{k}f_{k}-f_{k}e_{k})(\mathcal{D}T(\bar{v} + z))\\
=&e_{k} \mathcal{D}^{\ov}(f_{k}T(v + z))- f_{k} \mathcal{D}^{\ov}(e_{k}T(v + z))\\
=&e_{k} \mathcal{D}^{\ov}\left(\sum_{r=1}^{k} f_{kr}(v + z)T(v + z- \delta^{kr})\right)-
 f_{k} \mathcal{D}^{\ov}\left(\sum_{s=1}^{k}e_{ks}(v + z)T(v + z+\delta^{ks})\right)
\end{align*}
If $r\neq j$, then $e_kf_{kr}(v + z)T(v + z- \delta^{kr})\in  {\mathcal F}_{ij} \otimes {V}(T(v))$.
By Lemma \ref{dij-commute} one has that $e_{k} \Dv( f_{kr}(v + z)T(v + z- \delta^{kr})) =\Dv (e_{k}f_{kr}(v + z)T(v + z- \delta^{kr}))$. Similarly if If $s\neq i$, then
$f_{k} \Dv ( e_{ks}(v + z)T(v + z+ \delta^{ks})) =\Dv (f_{k}  e_{ks}(v + z)T(v + z+ \delta^{ks}))$. Thus
\begin{align*}
&(e_{k}f_{k}-f_{k}e_{k})(\mathcal{D}T(\bar{v} + z))\\
 =& \Dv (\sum_{r\neq j}e_{k} f_{kr}(v + z)T(v + z- \delta^{kr})
+e_{k} \Dv( f_{kj}(v + z)T(v + z- \delta^{kj}))\\
&- \Dv(\sum_{s\neq i}f_{k} e_{ks}(v + z)T(v + z+\delta^{ks})
-f_{k} \Dv( e_{ki}(v + z)T(v + z+ \delta^{ki})))
\end{align*}
The action of $e_kf_k-f_ke_k$ on $T(v+z)$ is as follows
\begin{align*}
(e_kf_k-f_ke_k)T(v + z)=&\sum_{r,s=1}^{k} f_{kr}(v+z)e_{ks}(v + z-\delta^{kr})T(v + z- \delta^{kr}+\delta^{ks})\\
&-\sum_{r,s=1}^{k} e_{ks}(v+z)f_{kr}(v + z+\delta^{ks})T(v + z+ \delta^{ks}-\delta^{kr})\\
=&h_k(v+z)T(v + z).
\end{align*}
one has that
$f_{kr}(v+z)e_{ks}(v + z-\delta^{kr})-e_{ks}(v+z)f_{kr}(v + z+\delta^{ks})=0$ when $r\neq s$, and
$\sum_{r=1}^{k} f_{kr}(v+z)e_{kr}(v + z-\delta^{kr})
-e_{kr}(v+z)f_{kr}(v + z+\delta^{ks})=h_k(v+z)$.
Then
\begin{align*}
&(e_{k}f_{k}-f_{k}e_{k})(\mathcal{D}T(\bar{v} + z))\\
=& \Dv\left(\sum_{r\neq j}\sum_{s\in\{i,r\}}f_{kr}(v + z) e_{ks}(v + z- \delta^{kr}) T(v + z- \delta^{kr}+\delta^{ks})\right)\\
 +& \Dv( f_{kj}(v + z)\Dv\left(\sum_{s=1}^{k}[x-y]_qe_{ks}(v + z- \delta^{kj})T(v + z- \delta^{kj}+\delta^{ks})\right)\\
 -&\Dv\left(\sum_{s\neq i}\sum_{r\in\{j,s\}} e_{ks}(v + z)f_{kr}(v + z+\delta^{ks}) T(v + z+\delta^{ks}-\delta^{kr})\right)\\
 -& \Dv( e_{ki}(v + z)\Dv\left(\sum_{r=1}^{k}[x-y]_qf_{kr}(v + z+ \delta^{ki})T(v + z+ \delta^{ki}-\delta^{kr})\right)
 \end{align*}
Now we consider the coefficients of $T(\ov + z+ \delta^{ks}-\delta^{kr})$ and
$\mathcal{D}T(\ov + z+ \delta^{ks}-\delta^{kr})$.
When $r\neq s$, $f_{kr}(v+z)e_{ks}(v + z-\delta^{kr})-e_{ks}(v+z)f_{kr}(v + z+\delta^{ks})=0$. If $r\neq i,j$, the coefficient of $T(\ov + z- \delta^{kr}+\delta^{ki})$ is $\Dv(f_{kr}(v + z) e_{ki}(v + z- \delta^{kr}))-
\Dv( e_{ki}(v + z))\Dv([x-y]_qf_{kr}(v + z+ \delta^{ki}))$. It follows from Lemma \ref{sym formula} that the coefficients of $T(\ov + z+ \delta^{ks}-\delta^{kr})$ and
$\mathcal{D}T(\ov + z+ \delta^{ks}-\delta^{kr})$ are zero.
\noindent
Similarly one has that the coefficient of $T(\ov  + z+\delta^{ks}-\delta^{kj})$ is zero when $s\neq i,j$.
By definition $\mathcal{D} T(\ov+ z- \delta^{kr}+\delta^{ki})=0$ if $r\neq i$ and $\mathcal{D} T(\ov+ z- \delta^{kj}+\delta^{ks})=0$ if $s\neq j$.
Since $T(\ov + z+ \delta^{ki}-\delta^{kj})=T(\ov + z)$,  $\D T(\ov + z+ \delta^{ki}-\delta^{kj})=\D T(\ov + z)$,
all the remaining tableaux are $T(\ov + z)$ and $\D T(\ov + z)$.

The coefficient of $T(\ov + z)$ is as follows
\begin{align*}
&\sum_{r\neq j}\Dv(f_{kr}(v + z) e_{kr}(v + z- \delta^{kr}))
-\sum_{s\neq i}\Dv( e_{ks}(v + z)f_{ks}(v + z+\delta^{ks}))\\
&+\Dv( f_{kj}(v + z))\sum_{s\in\{i,j\}}\Dv([x-y]_qe_{ks}(v + z- \delta^{kj}))\\
&-\Dv( e_{ki}(v + z))\sum_{r\in\{i,j\}}\Dv([x-y]_qf_{kj}(v + z+ \delta^{ki}))
\end{align*}
\begin{align*}
=&\Dv( e_{ki}(v + z)f_{ki}(v + z+\delta^{ki})-f_{kj}(v + z) e_{kj}(v + z- \delta^{kj}))\\
&+\Dv( f_{kj}(v + z))\Dv([x-y]_qe_{ki}(v + z- \delta^{kj}))\\
&+\Dv( f_{kj}(v + z))\Dv([x-y]_qe_{kj}(v + z- \delta^{kj}))\\
&-\Dv( e_{ki}(v + z))\Dv([x-y]_qf_{ki}(v + z+ \delta^{ki}))\\
&-\Dv( e_{ki}(v + z))\Dv([x-y]_qf_{kj}(v + z+ \delta^{ki}))
\end{align*}
\noindent
Since $ e_{ki}(v + z- \delta^{kj})= e_{kj}(v + z- \delta^{kj})^{\tau}$,
$ f_{kj}(v + z+ \delta^{ki})= f_{ki}(v + z+ \delta^{ki})^{\tau} $.
$f_{kj}(v+z)e_{ki}(v + z-\delta^{kj})-e_{ki}(v+z)f_{kj}(v + z+\delta^{ki})=0$
By lemma \ref{dv-formulas}  the coefficient of $T(\ov + z)$ is zero.

Since $\D T(\ov + z+ \delta^{ki}-\delta^{kj})=-\D T(\ov + z )$, the coefficient of $\mathcal{D}T(\ov + z)$ is

\begin{align*}
 &\sum_{r\neq j}\ev(f_{kr}(v + z) e_{kr}(v + z- \delta^{kr}) \mathcal{D} T(\ov+ z)\\
 &-\sum_{s\neq i}\ev( e_{ks}(v + z)f_{ks}(v + z+\delta^{ks})) \mathcal{D}T(\ov + z)\\
 &+ \Dv( f_{kj}(v + z)) \left({\rm ev} (\ov)([x-y]_qe_{kj}(v + z- \delta^{kj})-{\rm ev} (\ov)([x-y]_qe_{ki}(v + z- \delta^{kj})\right)\\
 &+\Dv\left( e_{ki}(v + z))\left({\rm ev} (\ov)([x-y]_qf_{kj}(v + z+ \delta^{ki})- {\rm ev} (\ov)([x-y]_qf_{ki}(v + z+ \delta^{ki})\right)\right)
 \end{align*}
 \begin{align*}
=&\ev(h_k(v+z))\\
 &+\ev( e_{ki}(v + z)f_{ki}(v + z+\delta^{ki})-f_{kj}(v + z) e_{kj}(v + z- \delta^{kj}))\\
 &+ \Dv( f_{kj}(v + z)) \left({\rm ev} (\ov)([x-y]_qe_{kj}(v + z- \delta^{kj})-{\rm ev} (\ov)([x-y]_qe_{ki}(v + z- \delta^{kj})\right)\\
 &+\Dv\left( e_{ki}(v + z))\left({\rm ev} (\ov)([x-y]_qf_{kj}(v + z+ \delta^{ki})- {\rm ev} (\ov)([x-y]_qf_{ki}(v + z+ \delta^{ki})\right)\right)
\end{align*}

By Lemma \ref{dv-formulas} one has that
\begin{align*}
 +&\ev( e_{ki}(v + z)f_{ki}(v + z+\delta^{ki})-f_{kj}(v + z) e_{kj}(v + z- \delta^{kj}))\\
 -& \Dv( f_{kj}(v + z)) {\rm ev} (\ov)([x-y]_qe_{ki}(v + z- \delta^{kj}))\\
 +&\Dv( f_{kj}(v + z) ){\rm ev} (\ov)([x-y]_qe_{kj}(v + z- \delta^{kj})\\
 -&\Dv( e_{ki}(v + z)){\rm ev} (\ov)([x-y]_qf_{ki}(v + z+ \delta^{ki}))\\
 +&\Dv( e_{ki}(v + z)){\rm ev} (\ov)([x-y]_qf_{kj}(v + z+ \delta^{ki}))
 =0
\end{align*}
Then the coefficient of $\D T(\ov+z)$ is $\ev(h_k(v+z))$. Thus
$(e_{k}f_{k}-f_{k}e_{k})(\D T(\ov + z))= h_k(\D T(\ov + z))$.

\emph{ Proof of $(v)$.} $(e_{r}^2e_{s}-(q+q^{-1})e_re_se_r+e_se_{r}^2)  (\D T(\ov + z))=0  \quad (|r-s|=1)$
If $r,s\neq k$ or $|z_{ki}-z_{kj}|\geq 2$, one has that $(e_{r}^2e_{s}-(q+q^{-1})e_re_se_r+e_se_{r}^2)  (\D T(\ov + z))=\Dv(e_{r}^2e_{s}-(q+q^{-1})e_re_se_r+e_se_{r}^2) T(v + z))=0 $. Suppose $r=k$, $s=k+1$, $|z_{ki}-z_{kj}|=1$. Without loos of generality, we assume that $z_{ki}=0,z_{kj}=1$. In the following formula we write $e_{rs}(v+z+w)$,$T(v+z+w_1+w_2+\cdots w_t)$, $\D T(v+z+w_1+\cdots,w_t)$ as $e_{rs}(w)$, $T(w_1,\ldots,w_t)$, $\D T(w_1,\ldots,w_t)$ respectively.

\begin{align*}
&(e_{k}^2e_{k+1}-(q+q^{-1})e_ke_{k+1}e_k+e_{k+1}e_{k}^2)  (\D T(\ov + z))\\
=&\sum_{r,t}\sum_{s\neq i}\Dv( e_{k+1,r}(0)e_{k,s}(\delta^{k+1,r})e_{k,t}(\delta^{k+1,r}+\delta^{k,s})
T(\delta^{k+1,r},\delta^{k,s},\delta^{k,t}))\\
&+\sum_{r,t}  \Dv(e_{k+1,r}(0)e_{k,i}(\delta^{k+1,r}))
\Dv([x-y]_qe_{k,t}(\delta^{k+1,r}+\delta^{k,i} )T(\delta^{k+1,r},\delta^{k,i},\delta^{k,t}))\\
&-(q+q^{-1}) \sum_{r,t}\sum_{s\neq i}\Dv(e_{k,s}(0)e_{k+1,r}(\delta^{k,s})e_{k,t}(\delta^{k,s}+\delta^{k+1,r})
T(\delta^{k,s},\delta^{k+1,r},\delta^{k,t}))\\
&-(q+q^{-1})\sum_{r,t} \Dv(e_{k,i}(0)e_{k+1,r}(\delta^{k,i}))
\Dv([x-y]_qe_{k,t}(\delta^{k,i}+\delta^{k+1,r})T(\delta^{k,i},\delta^{k+1,r},\delta^{k,t}))\\
&+\sum_{r,t}\sum_{s\neq i}
\Dv(e_{k,s}(0)e_{k,t}(\delta^{k,s})e_{k+1,r}(\delta^{k,s}+\delta^{k,t})T(\delta^{k,s},\delta^{k,t},\delta^{k+1,r}))\\
&+\sum_{r,t} \Dv(e_{k,i}(0))
\Dv([x-y]_qe_{k,t}(\delta^{k,s})e_{k+1,r}(\delta^{k,s}+\delta^{k,t})T(\delta^{k,s},\delta^{k,t},\delta^{k+1,r})).
\end{align*}
\noindent
Since $[e_k[e_k,e_{k+1}]]T(v)=0$, the coefficients of $T(\delta^{k,s}+\delta^{k,t}+\delta^{k+1,r})$ and $\D T(\delta^{k,s}+\delta^{k,t}+\delta^{k+1,r})$ are zero when $s,t\neq i$.

The coefficient of $T(\ov+z+\delta^{ki}+\delta^{ks}+\delta^{k+1,r})$, $s\neq i,j$ is
\begin{align*}
&\Dv (e_{k+1,r}(v+z)e_{k,s}(v+z+\delta^{k+1,r})e_{k,i}(\delta^{k+1,r}+\delta^{k,s}))\\
+&\Dv(e_{k+1,r}(v+z)e_{k,i}(v+z+\delta^{k+1,r}))\Dv([x-y]_qe_{k,s}(v+z+\delta^{k+1,r}+\delta^{k,i}))\\
-&(q+q^{-1}) \Dv( e_{k,s}(v+z)e_{k+1,r}(v+z+\delta^{k,s})e_{k,i}(v+z+\delta^{k,s}+\delta^{k+1,r}))\\
-&(q+q^{-1}) \Dv(e_{k,i}(v+z)e_{k+1,r}(v+z+\delta^{k,i}))\\
&\Dv([x-y]_qe_{k,s}(v+z+\delta^{k,i}+\delta^{k+1,r}))\\
+&\Dv( e_{k,s}(v+z)e_{k,i}(v+z+\delta^{k,s})e_{k+1,r}(v+z+\delta^{k,s}+\delta^{k,i}))\\
&\Dv(e_{k,i}(v+z))\Dv([x-y]_qe_{k,s}(v+z+\delta^{k,i})e_{k+1,r}(v+z+\delta^{k,i}+\delta^{k,s})
\end{align*}
Under the conditions $e_{k,s}(v+z+\delta^{k+1,r}+\delta^{k,i})$ is a symmetric function in $\mathcal{F}_{ij}$, so
\begin{align*}
\Dv([x-y]_qe_{k,s}(v+z+\delta^{k+1,r}+\delta^{k,i}))
&=\ev(e_{k,s}(v+z+\delta^{k+1,r}+\delta^{k,i}))),\\
\Dv( e_{k,s}(v+z+\delta^{k+1,r}+\delta^{k,i}))&=0
\end{align*}

which implies
\begin{multline*}
\Dv(e_{k+1,r}(v+z)e_{k,i}(v+z+\delta^{k+1,r}))\Dv([x-y]_qe_{k,s}(v+z+\delta^{k+1,r}+\delta^{k,i}))=\\
\Dv(e_{k+1,r}(v+z)e_{k,i}(v+z+\delta^{k+1,r}))e_{k,s}(v+z+\delta^{k+1,r}+\delta^{k,i}))
\end{multline*}

Similarly one has that
\begin{align*}
\Dv(e_{k,i}(v+z)e_{k+1,r}(v+z+\delta^{k,i}))
\Dv([x-y]_qe_{k,s}(v+z+\delta^{k,i}+\delta^{k+1,r}))\\
=\Dv(e_{k,i}(v+z)e_{k+1,r}(v+z+\delta^{k,i}))
e_{k,s}(v+z+\delta^{k,i}+\delta^{k+1,r})),\\
\Dv(e_{k,i}(v+z))\Dv([x-y]_qe_{k,s}(v+z+\delta^{k,i})e_{k+1,r}(v+z+\delta^{k,i}+\delta^{k,s})\\
=\Dv(e_{k,i}(v+z)) e_{k,s}(v+z+\delta^{k,i})e_{k+1,r}(v+z+\delta^{k,i}+\delta^{k,s}).
\end{align*}

Thus the coefficient of $T(\ov+z+\delta^{ki}+\delta^{ks}+\delta^{k+1,r})$ is zero.

The coefficient of $T(\ov+z+2\delta^{ki} +\delta^{k+1,r})$ is as follows
\begin{align*}
&\Dv( e_{k+1,r}(v+z)e_{k,j}(v+z+\delta^{k+1,r})e_{k,i}(v+z+\delta^{k+1,r}+\delta^{k,j})\\
+&\Dv(e_{k+1,r}(v+z)e_{k,i}(v+z+\delta^{k+1,r}))\Dv([x-y]_qe_{k,i}(v+z+\delta^{k+1,r}+\delta^{k,i}))\\
+&\Dv(e_{k+1,r}(v+z)e_{k,i}(v+z+\delta^{k+1,r}))\Dv([x-y]_qe_{k,j}(v+z+\delta^{k+1,r}+\delta^{k,i}))\\
-&(q+q^{-1}) \Dv(e_{k,j}(v+z)e_{k+1,r}(v+z+\delta^{k,j})e_{k,i}(v+z+\delta^{k,j}+\delta^{k+1,r})\\
-&(q+q^{-1}) \Dv(e_{k,i}(v+z)e_{k+1,r}(v+z+\delta^{k,i}))\Dv([x-y]_qe_{k,i}(v+z+\delta^{k,i}+\delta^{k+1,r}))\\
-&(q+q^{-1}) \Dv(e_{k,i}(v+z)e_{k+1,r}(v+z+\delta^{k,i}))\Dv([x-y]_qe_{k,j}(v+z+\delta^{k,i}+\delta^{k+1,r}))\\
+&\Dv(e_{k,j}(v+z)e_{k,i}(v+z+\delta^{k,j})e_{k+1,r}(v+z+\delta^{k,j}+\delta^{k,i})\\
+&\Dv(e_{k,i}(v+z))\Dv([x-y]_qe_{k,i}(v+z+\delta^{k,i})e_{k+1,r}(v+z+2\delta^{k,i} )\\
+&\Dv(e_{k,i}(v+z))\Dv([x-y]_qe_{k,j}(v+z+\delta^{k,i})e_{k+1,r}(v+z+\delta^{k,i}+\delta^{k,j} ).
\end{align*}
It follows from Lemma \ref{dv-formulas} that

\begin{align*}
&\Dv(e_{k+1,r}(v+z)e_{k,i}(v+z+\delta^{k+1,r}))\Dv([x-y]_qe_{k,i}(v+z+\delta^{k+1,r}+\delta^{k,i}))\\
+&\Dv(e_{k+1,r}(v+z)e_{k,i}(v+z+\delta^{k+1,r}))\Dv([x-y]_qe_{k,j}(v+z+\delta^{k+1,r}+\delta^{k,i}))\\
-&(q+q^{-1}) \Dv(e_{k,i}(v+z)e_{k+1,r}(v+z+\delta^{k,i}))\Dv([x-y]_qe_{k,i}(v+z+\delta^{k,i}+\delta^{k+1,r}))\\
-&(q+q^{-1}) \Dv(e_{k,i}(v+z)e_{k+1,r}(v+z+\delta^{k,i}))\Dv([x-y]_qe_{k,j}(v+z+\delta^{k,i}+\delta^{k+1,r}))\\
+&\Dv(e_{k,i}(v+z))\Dv([x-y]_qe_{k,i}(v+z+\delta^{k,i})e_{k+1,r}(v+z+2\delta^{k,i} )\\
+&\Dv(e_{k,i}(v+z))\Dv([x-y]_qe_{k,j}(v+z+\delta^{k,i})e_{k+1,r}(v+z+\delta^{k,i}+\delta^{k,j} )\\
=&\Dv ( e_{k+1,r}(v+z)e_{k,i}(v+z+\delta^{k+1,r}))e_{k,i}(v+z+\delta^{k+1,r}+\delta^{k,i}) \\
-&(q+q^{-1})e_{k,i}(v+z)e_{k+1,r}(v+z+\delta^{k,i}))e_{k,i}(v+z+\delta^{k,i}+\delta^{k+1,r}))\\
+& (e_{k,i}(v+z))e_{k,i}(v+z+\delta^{k,i})e_{k+1,r}(v+z+2\delta^{k,i})).
\end{align*}
Then the coefficient of $T(\ov+z+2\delta^{ki} +\delta^{k+1,r})$ is zero.
Similarly one has that the coefficients of $\D T(\ov+z+ \delta^{ki} +\delta^{ks}+\delta^{k+1,r}),s\neq i,j$ and
$\D T(\ov+z+2\delta^{ki} +\delta^{k+1,r})$  are zero.
Thus $(e_{k}^2e_{k+1}-(q+q^{-1})e_ke_{k+1}e_k+e_{k+1}e_{k}^2)  (\D T(\ov + z))=0$.
%%%%%%%%%%%%%%%%%%%%%%%%%%%%%%%%%%%%%%%%%%

\end{proof}

\end{document}